\documentclass[11pt]{article}

\usepackage[a4paper,margin=1in]{geometry}
\usepackage{thmtools}
\usepackage{amsmath,amssymb,amsthm,mathtools,mathrsfs}
\usepackage{enumitem}
\usepackage{microtype}
\usepackage[colorlinks=true,linkcolor=blue,citecolor=blue,urlcolor=blue]{hyperref}

\numberwithin{equation}{section}

\newtheorem{theorem}{Theorem}[section]

\newtheorem{conjecture}{Conjecture}
\newtheorem{proposition}[theorem]{Proposition}
\newtheorem{lemma}[theorem]{Lemma}
\newtheorem{corollary}[theorem]{Corollary}
\newtheorem{definition}[theorem]{Definition}
\newtheorem{remark}{Remark}[section]
\newtheorem{example}{Example}

\newcommand{\R}{\mathbb R}
\newcommand{\M}{\mathcal M}
\newcommand{\E}{\mathcal E}
\newcommand{\Per}{\mathcal P_{\rm er}}
\newcommand{\Lip}{\operatorname{Lip}}
\newcommand{\supp}{\operatorname{supp}}

\newcommand{\cL}{\mathscr L}
\newcommand{\cZ}{\mathscr Z}

\newcommand{\Max}{\M_{\max}}
\newcommand{\dist}{\operatorname{dist}}
\newcommand{\cl}{\operatorname{cl}}
\newcommand{\interior}{\operatorname{int}}

\newcommand{\cA}{\mathcal A}

\newcommand{\cS}{\mathcal S}
\newcommand{\cE}{\mathcal E}
\newcommand{\cAA}{\mathcal{AA}}
\newcommand{\gZ}{\mathcal Z}
\newcommand{\gPer}{\operatorname{Per}}

\makeatletter
\let\aa@address\@empty
\let\aa@emails\@empty
\newcommand{\address}[2][]{\gdef\aa@address{#2}}
\newcommand{\email}[1]{%
  \g@addto@macro\aa@emails{\noindent\mbox{\href{mailto:#1}{\texttt{#1}}}\par}}
\newcommand{\printauthorcontact}{%
  \par\bigskip
  \noindent\textit{Author address: }\textsuperscript{1}\,\aa@address\par
  \smallskip
  \noindent\textit{E-mail addresses:}\par
  \aa@emails}
\newcommand{\keywords}[1]{%
  \begingroup
  \renewcommand{\thefootnote}{}
  \begin{NoHyper}
  \footnotetext{\textit{Key words and phrases:} #1}
  \end{NoHyper}
  \endgroup}
\makeatother

\title{Maximizing Families and Typical Periodic Optimization\\
for Almost Additive Sequences}
\author{XiaoYu Zhang$^1$, Ercai Chen$^1$ and 
Xiaoyao Zhou$^1$\thanks{Corresponding author. E-mail address: \href{mailto:zhouxiaoyaodeyouxian@126.com}{\nolinkurl{zhouxiaoyaodeyouxian@126.com}}.}}
\address
{School of Mathematical Sciences and Institute of Mathematics,
Key Laboratory of NSLSCS, Ministry of Education, Nanjing Normal University,
Nanjing 210023, Jiangsu, P.R. China}
\email{zxiaoyu202405@163.com}
\email{ecchen@njnu.edu.cn}
\email{zhouxiaoyaodeyouxian@126.com}

\date{}
\begin{document}

\maketitle


\begin{abstract}
In this paper, we study typical periodic optimization (TPO) for almost
additive potentials in two perturbation spaces. Our main setting is the
Banach quotient $\mathcal E_{\rm orb}(X,T)$ of orbit-Lipschitz almost
additive potentials, where we extend the theory of maximizable sets and countable maximizable
families developed by W. Huang, O. Jenkinson, L. Xu and Y. Zhang
[Typical periodic optimization for dynamical systems: symbolic dynamics,
Invent. Math. 245 (2026), 1--63], and establish a global structural theorem.	
For a countable maximizable family, global TPO holds if every non boundary
member has $X$-extendable TPO and the boundary region has empty interior. 
As an application, we construct a compact system with global TPO for which
$\mathcal E_{\rm orb}(X,T)$ is infinite-dimensional and the maximizing
periods in open locking regions are unbounded. For a fixed almost additive
potential $\Phi$, we also develop relative TPO theory on its Lipschitz leaf. When $\Phi=0$, this framework reduces
to classical Lipschitz TPO. We prove the corresponding
leafwise structural theorem and give a non additive rank-one matrix example
on a full shift.
\end{abstract}

\section{Introduction}
Let $(X,T)$ be a topological dynamical system, where $X$ is a compact metric
space and $T:X\to X$ is continuous. Let $C(X,\mathbb R)$ be 
the Banach space of real-valued continuous functions on $X$, equipped
with the uniform norm
$
\|\cdot\|_\infty.
$ Denote by $\M(X,T)$   the set of
$T$-invariant Borel probability measures.  For any  $f\in C(X,\mathbb{R})$, the maximum
ergodic average is
$$
\beta(f):=\max_{\mu\in\M(X,T)}\int_X f\,d\mu,
$$
and the set of maximizing measures is
$$
\Max(f):=
\left\{
\mu\in\M(X,T):\int_X f\,d\mu=\beta(f)
\right\}.
$$
Ergodic optimization studies the  measures in $\Max(f)$ and their relation to
the dynamics of $T$.  Basic questions concern uniqueness, support, entropy,
and periodicity of maximizing measures; see
\cite{Jenkinson2006Survey,Jenkinson2019Survey,Garibaldi2017,Bochi2018}. For several classical
chaotic systems, maximizing measures of typical regular observables 
often have low dynamical complexity. In particular, they are often
supported on periodic orbits. This phenomenon leads to the periodic
optimization problem.

We say that $f$ has the \emph{periodic optimization property} if 
$
\Max(f)=\{\mu_Q\}
$ 
for some periodic orbit $Q$, where $\mu_Q$ denotes the invariant probability
measure supported on $Q$. Although periodic measures form a topologically
small subset of $\M(X,T)$, Hunt and
Ott \cite{HuntOtt1996PRL,HuntOtt1996PRE},  first pointed out that, for
chaotic systems,  that regular functions (e.g. smooth, or Lipschitz) often \emph{typically}
have this property. 
Hunt and Ott formulated this phenomenon in a probabilistic form. Yuan and
Hunt later proposed the following topological conjecture.
\begin{conjecture}[{\cite{YuanHunt1999}}]
	Let $(X,T)$ be an Axiom~A or uniformly expanding system. Then a
	topologically typical smooth function admits a maximizing measure supported
	on a periodic orbit.
\end{conjecture}

Progress towards these conjecture was made in several settings.
Bousch \cite{Bousch2000,Bousch2001,Bousch2008} developed calibrated
subactions for expanding maps and shift spaces. Contreras, Lopes and
Thieullen \cite{ContrerasLopesThieullen2001} studied minimizing measures
for expanding circle maps. Morris \cite{Morris2008,Morris2009} obtained
generic zero-entropy results for H\"older maximizing measures and a
Ma\~n\'e cohomology lemma for intermittent maps. Quas and Siefken
\cite{QuasSiefken2012} studied periodic optimization for supercontinuous
functions on shift spaces. A major breakthrough was made by Contreras \cite{Contreras2016}.
Building on earlier advances
\cite{ContrerasLopesThieullen2001,QuasSiefken2012,YuanHunt1999},
he proved \emph{typical periodic optimization} (TPO) for open
distance-expanding maps: there exists an open and dense subset
$$
\mathcal P\subset\Lip(X)
$$
such that every $f\in\mathcal P$ has the periodic optimization
property. Huang, Lian,
Ma, Xu and Zhang \cite{HuangLianMaXuZhang} later extended this result to
uniformly hyperbolic systems. Li and Zhang \cite{LiZhang2025} obtained a
related result for expanding Thurston maps. A common feature of these
approaches is a combination of a closing lemma, shadowing, and
perturbations based on a Ma\~n\'e cohomology lemma.

Huang, Jenkinson, Xu and Zhang
\cite{HuangJenkinsonXuZhang2026} recently developed the theory of
maximizable families. This theory was developed to study
TPO for weakly hyperbolic
systems in which the Ma\~n\'e cohomology lemma fails. It reduces the
periodic optimization problem to periodic optimization on the members
of a countable maximizable family. In particular, it proves TPO for
many symbolic systems, including all sofic shifts and all $S$-gap
shifts.

For comparison, consider $C(X,\mathbb R)$ endowed with the uniform
norm. The uniform topology induced on $\Lip(X)$ is weaker than the
Lipschitz topology. Jenkinson \cite{Jenkinson2006Unique} showed that
an ergodic measure can be realized as the unique maximizing measure
of a   continuous function, while Morris \cite{Morris2010}
proved a generic uniqueness result for continuous potentials. These
results concern uniqueness rather than TPO. Indeed, generic
uniqueness in the uniform topology does not provide an open and dense
set in the Lipschitz topology, and it does not ensure that the unique
maximizing measure is periodic. Thus it does not imply TPO on
$\Lip(X)$. This difference is particularly important for non-additive
potentials, where the ambient space and its norm must be specified
before  ``open'', ``dense'', and ``typical'' have a definite
meaning.

Many quantities arising in non-conformal dynamics and matrix products cannot
be represented by the Birkhoff sums of a single continuous function. Instead,
they are described by a sequence
$
\Phi=\{\phi_n\}_{n\geq1}\subset C(X,\mathbb R),
$
which need not satisfy the additive relation
$
\phi_{n+m}=\phi_n+\phi_m\circ T^n.
$
The first general form of non-additive thermodynamic formalism was introduced
by Falconer \cite{Falconer1988} for subadditive potentials arising from
non-conformal repellers.  Barreira \cite{Barreira1996} developed a broader
non-additive formalism and applied it to dimension theory.  The almost
additive formalism was then developed by Barreira \cite{Barreira2006} and
Mummert \cite{Mummert2006}; see also the survey
\cite{Barreira2010}.  Earlier matrix examples were studied by Feng and Lau
\cite{FengLau2002}.  The thermodynamic formalism for subadditive potentials
was further developed by Cao, Feng and Huang \cite{CaoFengHuang2008}.
Feng and Huang \cite{FengHuang2010} introduced asymptotically additive
potentials in their study of Lyapunov spectra.

An almost additive sequence $
\Phi=\{\phi_n\}_{n\geq1}\subset C(X,\mathbb{R})$  satisfies
$$
\left\|
\phi_{n+m}-\phi_n-\phi_m\circ T^n
\right\|_\infty\leq C_\Phi
$$
for all $n,m\geq1$.  A sequence is asymptotically additive if it can be approximated by Birkhoff sums of continuous functions with sublinear error, and every almost additive sequence is asymptotically additive. The set of asymptotically additive sequences of observables is much larger than
those formed by sequences of Birkhoff averages. Yet, Cuneo \cite{Cuneo2020} showed that every asymptotically additive sequence can be approximated by the Birkhoff sums of some $g\in C(X,\mathbb{R})$ with sublinear error:
$$
\lim_{n\to\infty}\frac1n\|\phi_n-S_ng\|_\infty=0,
$$ 
where $S_ng:=\sum_{j=0}^{n-1}g(T^jx)$. 
This result reduces many continuous non-additive questions to the additive
case.  

Along with the development of non-additive thermodynamic formalism,
ergodic optimization for non-additive potentials has also received
considerable attention.  Morris
\cite{Morris2013} considered Mather sets for matrix sequences.  Garibaldi and
Gomes \cite{GaribaldiGomes2016} studied Aubry sets for asymptotically
subadditive potentials.  Zhao \cite{Zhao2019} studied constrained ergodic
optimization for asymptotically additive potentials.  Bomfim, Huo, Varandas
and Zhao \cite{BomfimHuoVarandasZhao2023} introduced a Banach quotient
$\cS$ of asymptotically additive potentials.  They proved generic uniqueness,
and obtained generic results on entropy and support. 

This leads to a natural question: how should TPO be defined for almost
additive potentials? Since TPO concerns open and dense sets, one must
first choose a perturbation space and a norm. In this paper we use two
spaces with different roles.
The first is the global quotient $\cE_{\rm orb}$.  TPO on $\cE_{\rm orb}$ is a global non-additive
statement and is the main form of TPO studied in this paper.

Let $\cA(X,T)$ be the Banach space of asymptotically additive sequences with
the strong sequence norm $\|\cdot\|_{\mathcal A}$.  For
$\Phi=\{\phi_n\}_{n\geq1}$, let $D(\Phi)$ be its almost additivity bound and
let $L_{\rm orb}(\Phi)$ be its orbit Lipschitz seminorm.  We consider
$$
\mathcal{AA}_{\rm orb}(X,T)
:=
\left\{
\Phi\in\cA(X,T):D(\Phi)<\infty,
L_{\rm orb}(\Phi)<\infty
\right\}
$$
with the norm
$$
\|\Phi\|_{\mathcal{AA}_{\rm orb}}
:=
\|\Phi\|_{\cA}+D(\Phi)+L_{\rm orb}(\Phi).
$$
Let $\mathcal N$ be the space of uniformly sublinear sequences and define
$$
\cE_{\rm orb}(X,T)
:=
\mathcal{AA}_{\rm orb}(X,T)
\big/
\bigl(\mathcal{AA}_{\rm orb}(X,T)\cap\mathcal N\bigr).
$$
If $\eta=[\Phi]\in\cE_{\rm orb}(X,T)$ and
$\mu\in\M(X,T)$, then
$
\eta_*(\mu)
:=
\Phi_*(\mu)
=
\lim_{n\to\infty}\frac1n\int_X\phi_n\,d\mu
$
is well defined.  We write
$
\Max(\eta)
:=
\left\{
\mu\in\M(X,T):
\eta_*(\mu)=\max_{\nu\in\M(X,T)}\eta_*(\nu)
\right\}.
$

Our first result gives the global Banach space used in the paper.
\begin{restatable}
	{introtheorem}{globalspace}
	\label{thm:intro-global-space}
	The space
	\((\mathcal{AA}_{\rm orb}(X,T),
	\|\cdot\|_{\mathcal{AA}_{\rm orb}})\) is a Banach space, and it is continuously and densely embedded into
	\((\mathcal A(X,T),\|\cdot\|_{\mathcal A})\).  In
	particular,
	\begin{equation*}\label{eq:orb-dense-A}
		\overline{\mathcal{AA}_{\rm orb}(X,T)}^{\,
			\|\cdot\|_{\mathcal A}}
		=\mathcal A(X,T).
	\end{equation*}
	The quotient
	$\cE_{\rm orb}(X,T)$ is Banach, and
	its natural map into the asymptotic quotient $\cS$ is continuous and
	injective, with dense image.
\end{restatable}

We say that $(X,T)$ has \emph{global TPO} if an open dense subset of
$\cE_{\rm orb}(X,T)$ consists of classes with a unique periodic maximizing
measure.  The global structural theorem uses countable maximizable families.  For a
closed invariant set $Z\subset X$, put
$$
\cE_Z
:=
\{\eta\in\cE_{\rm orb}(X,T):
\Max(\eta)\cap\M(Z,T)\neq\varnothing\}.
$$
A countable family $\{Z_i\}$ is $\cE_{\rm orb}$-maximizable if
$$
\cE_{\rm orb}(X,T)=\bigcup_i\cE_{Z_i}.
$$
For a subsystem $Z\subset X$, the restriction of global classes gives the
extendable space $\cE_{\rm orb}^{X}(Z,T)$.  TPO in this space is called
$X$-extendable TPO. We give the following Global structural theorem.
\begin{restatable}
		{introtheorem}{globalstructural}
		\label{thm:intro-global-structural}
	Let
	$
	\mathcal Z=\{Z_0\}\cup\{Z_i:i\geq1\}
	$
	be a countable $\cE_{\rm orb}(X,T)$-maximizable family.  Suppose that
	$(Z_i,T)$ has $X$-extendable TPO for every $i\geq1$.  Let
	$$
	\mathcal P
	:=
	\interior
	\{\eta\in\cE_{\rm orb}(X,T):
	\Max(\eta)=\{\mu_Q\}
	\text{ for some periodic orbit }Q\}.
	$$
	Then
	$$
	\mathcal P\cup\interior(\cE_{Z_0})
	$$
	is open and dense in $\cE_{\rm orb}(X,T)$.  In particular, if
	$
	\interior(\cE_{Z_0})=\varnothing,
	$
	then $(X,T)$ has global TPO.
\end{restatable}

The global theory is not only formal.  We give an infinite-dimensional example
with global TPO.  It is built from a countable family of periodic orbits whose
periods tend to infinity.

\begin{example}
	\label{thm:intro-bouquet}
	There exists a compact metric dynamical system $(X,T)$ with periodic orbits
	$
	Z_0,Z_1,Z_2,\ldots,
	$
	where $Z_0$ is a fixed point and $Z_m$ has prime period $m+1$ for
	$m\geq1$, such that the following statements hold:
	\begin{enumerate}[label=\textup{(\roman*)}]
		\item Every invariant measure has a unique form
		$$
		\mu=\sum_{m=0}^{\infty}w_m\mu_m,
		\qquad
		w_m\geq0,
		\qquad
		\sum_{m=0}^{\infty}w_m=1,
		$$
		where $\mu_m$ is the periodic measure on $Z_m$.
		\item The family $\{Z_m:m\geq0\}$ is a countable
		$\cE_{\rm orb}$-maximizable family.
		\item The set
		$$
		\bigcup_{m=0}^{\infty}\interior(\cE_{Z_m})
		$$
		is open and dense in $\cE_{\rm orb}(X,T)$.  Every class in this set has a
		unique periodic maximizing measure.
		\item For every $m\geq0$, there is a non-empty open set on which $\mu_m$
		is the unique maximizing measure.  Hence the periods in these open sets are
		not bounded.
		\item The Banach space $\cE_{\rm orb}(X,T)$ is infinite-dimensional.
		Moreover, for every $m\geq0$ there is an orbit Lipschitz almost additive
		sequence which is not additive and whose class has $\mu_m$ as a locked
		unique maximizing measure.
	\end{enumerate}
\end{example}

The second space is a fixed Lipschitz leaf.  For a fixed almost additive
sequence $\Phi$, set
$
\cL_\Phi(X)
:=
\{\Phi+S_\bullet f:f\in\Lip(X)\}.
$
The leaf is identified with $\Lip(X)$.  Only the additive Lipschitz part
$S_\bullet f$ varies.  This gives a relative TPO theory for the fixed
 $\Phi$.     If
$\Phi=0$, relative TPO is exactly classical TPO.  A leafwise result is not a
global result in $\cE_{\rm orb}$, since openness and density are taken in
different spaces.

We now fix an almost additive sequence
$\Phi=\{\phi_n\}_{n\geq1}$ and  $f\in\Lip(X)$.  Set 
$$
\Phi^f:=\Phi+S_\bullet f. 
$$
We say that $(X,T)$ has \emph{relative $\Phi$-TPO} if there is an open dense
set $\mathcal P^\Phi\subset\Lip(X)$ such that
$$
\Max(X,T,\Phi^f)=\{\mu_Q\}
$$
for every $f\in\mathcal P^\Phi$ and some periodic orbit $Q$ depending on
$f$.  The topology is the Lipschitz topology of the fixed leaf
$\cL_\Phi(X)$.

\begin{restatable}
	{introtheorem}{leafwisestructural}
	\label{thm:intro-leafwise-structural}
	Let
	$
	\mathcal Z=\{Z_0\}\cup\{Z_i:i\geq1\}
	$
	be a countable $\Phi$-leafwise maximizable family.  Suppose that
	$(Z_i,T)$ has relative $(\Phi|_{Z_i})$-TPO for every $i\geq1$.  Then
	$$
	\mathcal P^\Phi\cup\interior(\Lip_{Z_0}^\Phi)
	$$
	is open and dense in $\Lip(X)$.  Hence $(X,T)$ has relative $\Phi$-TPO if
	$$
	\interior(\Lip_{Z_0}^\Phi)=\varnothing.
	$$
	The same conclusion holds when the family has no boundary set $Z_0$.
\end{restatable}

The leafwise subsystem reduction uses the McShane extension theorem.  A small
Lipschitz perturbation on $Z$ extends to a small Lipschitz perturbation on
$X$.  
If there exists $f_0\in\Lip(Z)$ such that
$$
\lim_{n\to\infty}\frac1n
\|\phi_n|_Z-S_nf_0\|_{\infty,Z}=0,
$$
then relative $(\Phi|_Z)$-TPO is the same as classical Lipschitz TPO on
$Z$, after the translation $f\mapsto f_0+f$.  This gives a direct way to
obtain leafwise examples from classical ones.

\begin{example}
	\label{thm:intro-matrix}
	Let $(\Sigma_k,\sigma)$ be the one-sided full shift on $k\geq2$ symbols.
	Let $C=(c_{ij})_{1\leq i,j\leq k}$ be a positive matrix.  Put
	$$
	u_i=e_i,
	\qquad
	v_i=(c_{i1},\ldots,c_{ik})^{\mathsf T},
	\qquad
	A_i=u_iv_i^{\mathsf T}.
	$$
	For $x=(x_j)_{j\geq0}\in\Sigma_k$, define
	$$
	A^{(n)}(x):=A_{x_{n-1}}\cdots A_{x_0},
	\qquad
	\phi_n^C(x):=\log\|A^{(n)}(x)\|,
	$$
	and write $\Phi_C=\{\phi_n^C\}_{n\geq1}$.  Then $\Phi_C$ is orbit
	Lipschitz and almost additive, but it is not additive.  Moreover,
	$(\Sigma_k,\sigma)$ has relative $\Phi_C$-TPO.  Hence there exists an open
	dense set $G\subset\Lip(\Sigma_k)$ such that, for every
	$f\in G$,
	$$
	\Max(\Sigma_k,\sigma,\Phi_C+S_\bullet f)=\{\mu_Q\}
	$$
	for some periodic orbit $Q$.
\end{example}

The paper is arranged as follows.  We first introduce almost additive
potentials, maximizing measures, Lipschitz leaves, and the global spaces
$\mathcal{AA}_{\rm orb}$ and $\cE_{\rm orb}$.  We then extend the theory of
maximizable sets, minimax sets, completely maximizing sets, and the subset
closing property.  The global TPO theory is developed next. We then give an infinite-dimensional example of global TPO.  Finally, we develop relative TPO on a
fixed Lipschitz leaf and give the  example.

\section{Preliminary}
In this section we  introduces almost additive potentials and their maximizing
measures,  followed by a comparison of three related spaces:  the
Banach space of asymptotically additive sequences $\mathcal A(X,T)$, its asymptotic
quotient $\mathcal S$, and the  Lipschitz leaf $\cL_\Phi(X)$. We also construct a 
Banach space of orbit Lipschitz almost additive sequences $\mathcal{AA}_{\rm orb}(X,T) $ and its 
quotient space $\mathcal E_{\rm orb}$. These objects support two complementary formulations of
typical periodic optimization developed later in the paper. 

Let $(X,T)$ be a topological dynamical system (TDS), 
The compact convex set of \(T\)-invariant Borel probability measures is denoted by
\(\M(X,T)\), and its set of ergodic measures by \(\E(X,T)\).  The support of a
measure \(\mu\) is denoted by \(\supp\mu\).  If \(Y\subset X\) is closed and
invariant, then
$ \M(Y,T):=\{\mu\in\M(X,T):\supp\mu\subset Y\}.$
Let $\Lip(X)$ be the Banach space of all real-valued Lipschitz functions on $X$, equipped with the norm 
$$
 \|f\|_{\Lip}:=\|f\|_\infty+|f|_{\Lip},
 \qquad
 |f|_{\Lip}:=\sup_{x\neq y}\frac{|f(x)-f(y)|}{d(x,y)}.
$$

\begin{definition}
A sequence \(\Phi=\{\phi_n\}_{n\geq1}\subset C(X,\mathbb{R})\) is \emph{almost additive} if
there exists \(C_\Phi\geq0\) such that, for all \(n,m\geq1\) and \(x\in X\),
\begin{align}\label{eq:aa}
 -C_\Phi+\phi_n(x)+\phi_m(T^nx)
 \leq \phi_{n+m}(x)
 \leq C_\Phi+\phi_n(x)+\phi_m(T^nx).
\end{align}
\end{definition}

For \(\mu\in\M(X,T)\), the  sequence
\(a_n(\mu):=\int\phi_n\,d\mu\) is almost additive.  Applying Fekete's lemma to
\(a_n+C_\Phi\) and to \(-a_n+C_\Phi\) gives the following.

\begin{lemma}\label{lem:functional}
Let $(X,T)$ be a TDS and let
$\Phi=\{\phi_n\}_{n\geq1}\subset C(X,\mathbb R)$
be an almost additive sequence and every \(\mu\in\M(X,T)\), the limit
\begin{equation*}\label{eq:Phi-star}
 \Phi_*(\mu):=\lim_{n\to\infty}\frac1n\int\phi_n\,d\mu
\end{equation*}
exists.  Moreover, for every \(k\geq1\),
\begin{equation}\label{eq:uniform-approx}
 \left|\Phi_*(\mu)-\frac1k\int\phi_k\,d\mu\right|
 \leq \frac{C_\Phi}{k}.
\end{equation}
Consequently, \(\Phi_*:\M(X,T)\to\R\) is continuous and affine.
If \(\mu\) is ergodic, then
\begin{equation*}\label{eq:aa-ergodic}
 \lim_{n\to\infty}\frac1n\phi_n(x)=\Phi_*(\mu)
 \quad\text{for \(\mu\)-a.e. }x\in X.
\end{equation*}
\end{lemma}
\begin{proof}
	Every almost additive sequence is asymptotically additive. Hence the
 	limit for $\Phi_*(\mu)$ and the a.e. pointwise
	convergence above follow from \cite{FengHuang2010}. 
	
	Set
	$
	a_n(\mu):=\int\phi_n\,d\mu.
	$
	For a fixed $k\geq1$, repeated use of almost additivity and the
	$T$-invariance of $\mu$ gives
	$$
	\left|a_{qk}(\mu)-q a_k(\mu)\right|
	\leq(q-1)C_\Phi
	$$
	for every $q\geq1$. Dividing by $qk$ and letting $q\to\infty$ yields
	\eqref{eq:uniform-approx}. 	
	Since this estimate is uniform in $\mu$, the map $\Phi_*$ is the
	uniform limit of the continuous affine maps
	$
	\mu\to\frac1k\int\phi_k\,d\mu.
	$
	Therefore, $\Phi_*$ is continuous and affine.
\end{proof}

We next introduce the maximum ergodic average and maximizing measures for almost additive sequence.
\begin{definition}
Let $(X,T)$ be a TDS and let
$\Phi=\{\phi_n\}_{n\geq1}\subset C(X,\mathbb R)$
be an almost additive sequence , its maximum ergodic average is defined by
\begin{align*}
 \beta(X,T,\Phi)=\beta(\Phi)&:=\sup_{\mu\in\M(X,T)}\Phi_*(\mu)=\max_{\mu\in\M(X,T)}\Phi_*(\mu),\\
 \Max(X,T,\Phi)=\Max(\Phi)&:=\{\mu\in\M(X,T):\Phi_*(\mu)=\beta(\Phi)\}.
\end{align*}
Any $\mu\in \Max(X,T,\Phi)$  is called a maximizing measure for $\Phi$, or a $\Phi$-maximizing measure.
By Lemma~\ref{lem:functional}, the maximum  is attained because $\Phi_*$ is continuous on the weak$^*$ compact space $\mathcal M(X,T)$. Hence $\Max(\Phi)$ is nonempty, compact, and convex.
\end{definition}

\begin{remark}
	We can normalize by  
	$\widehat\phi_n:=\phi_n-n\beta(X,T,\Phi).$
	Then $\widehat\Phi=\{\widehat\phi_n\}_{n\geq 1}$ has the same almost additivity bound and the same
	maximizing measures, while \(\beta(X,T,\widehat\Phi)=0\). Thus, we may assume $\beta(\Phi)=0$ without loss of generality.
\end{remark}

To construct the perturbation space used in the global theory, we first
recall asymptotically additive sequences. 

A sequence
$\Psi=\{\psi_n\}_{n\geq1}\subset C(X,\mathbb{R})$ is called asymptotically additive if,
for every $\xi>0$, there exists $g_\xi\in C(X,\mathbb{R})$ such that
\begin{equation*}
	\limsup_{n\to\infty}
	\frac1n\|\psi_n-S_ng_\xi\|_\infty\leq\xi.
\end{equation*}
We denote by $\mathcal A(X,T)$ the vector space of all asymptotically
additive sequences.

\begin{remark}\label{rem:cuneo}
	Cuneo \cite{Cuneo2020} proved that for any asymptotically additive sequences $\Phi=\{\phi_n\}_n$, there exists \(f\in C(X,\mathbb{R})\) such that
	$
		\lim_{n\to\infty}\frac1n\|\phi_n-S_nf\|_\infty=0.
	$
	Hence \(\Phi_*(\mu)=\int f\,d\mu\) for  \(\mu\in\mathcal{M}(X,T)\). 
\end{remark}

Define the closed subspace
$
\mathcal N
:=\left\{\Psi\in\mathcal A(X,T):
\limsup_{n\to\infty}\frac1n\|\psi_n\|_\infty=0\right\}
$
and the quotient
$
\mathcal S:=\mathcal A(X,T)/\mathcal N.
$
Its norm is
\begin{equation*}\label{eq:S-norm}
	\|[\Psi]\|_{\mathcal S}
	=\limsup_{n\to\infty}\frac1n\|\psi_n\|_\infty.
\end{equation*}
We next introduce the orbit Lipschitz sequences and the space consisting of these sequences.
For \(n\geq1\), $x,y\in X$, set
\begin{equation*}\label{eq:orbit-l1-metric}
	d_n^{(1)}(x,y):=\sum_{j=0}^{n-1}d(T^jx,T^jy).
\end{equation*}
For a sequence \(\Psi=\{\psi_n\}_{n\geq 1}\), define its almost additivity bound and
orbit Lipschitz seminorm by
\begin{align} \label{eq:orb-seminorm}
	D(\Psi)
	:=\sup_{n,m\geq1}
	\|\psi_{n+m}-\psi_n-\psi_m\circ T^n\|_\infty,
	\qquad
	L_{\rm orb}(\Psi)
	:=\sup_{\substack{n\geq1\\x\neq y}}
	\frac{|\psi_n(x)-\psi_n(y)|}{d_n^{(1)}(x,y)}.
\end{align}
For any $x\neq y$, the definition of $d_n^{(1)}$  yields $d_n^{(1)}(x,y)>0.$

We denote by $\mathcal{AA}_{\rm orb}(X,T)$ the space of all orbit
Lipschitz almost additive sequences, namely 
\begin{equation*}\label{eq:AA-orb-space}
	\mathcal{AA}_{\rm orb}(X,T)
	:=\{\Psi\in\mathcal A(X,T):D(\Psi)<\infty,
	L_{\rm orb}(\Psi)<\infty\}.
\end{equation*}
We equip this space with the norm
\begin{equation*}
	\|\Psi\|_{\mathcal{AA}_{\rm orb}}
	:=\|\Psi\|_{\mathcal A}+D(\Psi)+L_{\rm orb}(\Psi).
\end{equation*}

The following result shows that orbit Lipschitz almost additive sequences
form a complete perturbation space and are dense among all asymptotically
additive sequences in the $\|\cdot\|_{\mathcal A}$-norm.
\begin{theorem}
	\label{thm:orb-banach-dense}
	The space
	\((\mathcal{AA}_{\rm orb}(X,T),
	\|\cdot\|_{\mathcal{AA}_{\rm orb}})\) is a Banach space, and it is continuously and densely embedded into
	\((\mathcal A(X,T),\|\cdot\|_{\mathcal A})\).  In
	particular,
	\begin{equation*}\label{eq:orb-dense-A}
		\overline{\mathcal{AA}_{\rm orb}(X,T)}^{\,
			\|\cdot\|_{\mathcal A}}
		=\mathcal A(X,T).
	\end{equation*}
\end{theorem}

\begin{proof}
	We first prove completeness.  Let \(\{\Psi^k\}_k\) be a Cauchy sequence in
	\(\|\cdot\|_{\mathcal{AA}_{\rm orb}}\).  Since \(\mathcal A(X,T)\) is Banach,
	there exists \(\Psi\in\mathcal A(X,T)\) such that
	$ \|\Psi^k-\Psi\|_{\mathcal A}\to0.$ 
	For each  $n$, 
	$$
	\frac{1}{n}\|\psi^k_n-\psi_n\|_{\infty}\leq \|\Psi^k-\Psi\|_{\mathcal A}\to0.
	$$
	Thus, \(\psi_n^k\) converges uniformly to \(\psi_n\) in \(C(X,\mathbb{R})\).  
	Since $\{\Psi^k\}_k$ is Cauchy in
	$\mathcal{AA}_{\rm orb}(X,T)$, we have $$D(\Psi^p-\Psi^q)\to0,\quad(p,q\to\infty).$$ For any $\varepsilon>0$, there exists $N>0$ such that for $p,q\geq N$, $$\|(\psi^p_{n+m}-\psi^p_n-\psi^p_m\circ T^n)-(\psi^q_{n+m}-\psi^q_n-\psi^q_m\circ T^n)\|_{\infty}\leq\varepsilon.$$ Letting $p\to\infty$ and then taking the supremum over $m,n$ to get $ D(\Psi^q-\Psi)\leq\varepsilon$. Since this holds for every $q\geq N$, we have, 
	$$
	D(\Psi^q-\Psi)\to0
	\qquad\text{as }q\to\infty.
	$$
	The same argument applied to the orbit Lipschitz seminorm gives
	$
	L_{\rm orb}(\Psi^k-\Psi)\to0.
	$
	Since $\Psi^q\in\mathcal{AA}_{\rm orb}(X,T)$ for every $q$, the
	preceding estimates imply that $D(\Psi)$ and
	$L_{\rm orb}(\Psi)$ are finite. Moreover,
	$
	\|\Psi^q-\Psi\|_{\mathcal{AA}_{\rm orb}}
	\to0.
	$

	We next prove the density.  Fix \(\Phi=(\phi_n)\in\mathcal A(X,T)\) and
	\(\varepsilon>0\).  By Cuneo's theorem, there exists \(g\in C(X,\mathbb{R})\) such that
	\begin{equation*}\label{eq:density-cuneo}
		\lim_{n\to\infty}\frac1n\|\phi_n-S_ng\|_\infty=0.
	\end{equation*}
	Choose $N$ large enough such that, for every $n\geq N$,
	$
	\|\phi_n-S_ng\|_\infty/n<\varepsilon/3$.  
	Since $\Lip(X)$ is dense in $C(X,\mathbb{R})$, choose $f\in\Lip(X)$ such that 
	\(\|f-g\|_\infty<\varepsilon/3\). Again using the density of $\Lip(X)$ in $C(X,\mathbb{R})$, for each
	$1\leq n<N$, choose $h_n\in\Lip(X)$ such that
	$
	\|h_n-\phi_n\|_\infty<n\varepsilon.
	$ 
	Define \(\Gamma=(\gamma_n)\) by
	\begin{equation*}\label{eq:repaired-sequence}
		\gamma_n:=
		\begin{cases}
			h_n,&1\leq n<N,\\
			S_nf,&n\geq N.
		\end{cases}
	\end{equation*}
	For \(n\geq N\),
	$$
	\frac1n\|\phi_n-\gamma_n\|_\infty
	\leq \frac1n\|\phi_n-S_ng\|_\infty+\|g-f\|_\infty
	<\frac{2\varepsilon}{3},
	$$
	while, for $1\leq n<N$,
	$$
	\frac1n\|\phi_n-\gamma_n\|_\infty
	=
	\frac1n\|\phi_n-h_n\|_\infty
	<\varepsilon.
	$$
	 Hence $\|\Phi-\Gamma\|_{\mathcal A}<\varepsilon.$
	
	To check almost additivity, write \(r_n:=\gamma_n-S_nf\).  Then \(r_n=0\) for
	\(n\geq N\), and
	$
	M:=\max_{1\leq n<N}\|r_n\|_\infty<\infty.
	$ 
	The additive parts cancel, giving
	$
	\gamma_{n+m}-\gamma_n-\gamma_m\circ T^n
	=r_{n+m}-r_n-r_m\circ T^n.
	$
	Thus \(D(\Gamma)\leq3M\).  Finally, for \(n<N\),
	$$
	|h_n(x)-h_n(y)|
	\leq |h_n|_{\Lip}d(x,y)
	\leq |h_n|_{\Lip}d_n^{(1)}(x,y),
	$$
	whereas for \(n\geq N\),
	$
	|S_nf(x)-S_nf(y)|
	\leq |f|_{\Lip}d_n^{(1)}(x,y).
	$
	Therefore
	$$
	L_{\rm orb}(\Gamma)
	\leq\max\{|f|_{\Lip},|h_1|_{\Lip},\ldots,|h_{N-1}|_{\Lip}\}<\infty.
	$$
	Hence $\Gamma\in\mathcal{AA}_{\rm orb}(X,T)$, and the preceding
	estimate proves density. Finally, the inclusion
	$
	\mathcal{AA}_{\rm orb}(X,T)\hookrightarrow\mathcal A(X,T)
	$
	is continuous because
	$
	\|\Psi\|_{\mathcal A}
	\leq
	\|\Psi\|_{\mathcal{AA}_{\rm orb}}
	$ 
	for every $\Psi\in\mathcal{AA}_{\rm orb}(X,T)$.
\end{proof}
We now take the quotient by uniformly sublinear sequences, thereby obtaining the Banach space employed in the global TPO theory.
\begin{corollary}
	\label{cor:orb-quotient}
	Let 
	$
		\mathcal E_{\rm orb}
		:=\mathcal{AA}_{\rm orb}(X,T)
		\big/
		\bigl(\mathcal{AA}_{\rm orb}(X,T)\cap\mathcal N\bigr), 
	$
	with its quotient norm.  Then $\mathcal E_{\rm orb}$ is Banach and the
	natural map
	$\mathcal E_{\rm orb}\to\mathcal S $ 
	is continuous, injective and has dense image.  If
	\(\eta=[\Phi]\in\mathcal E_{\rm orb}\), then
	$$
	\eta_*(\mu):=\Phi_*(\mu),\qquad \mu\in\M(X,T),
	$$
	is well defined and
	\begin{equation*}
		\sup_{\mu\in\M(X,T)}|\eta_*(\mu)|
		\leq\|\eta\|_{\mathcal S}
		\leq\|\eta\|_{\mathcal E_{\rm orb}}.
	\end{equation*}
\end{corollary}

\begin{proof}
	The intersection \(\mathcal{AA}_{\rm orb}(X,T)\cap\mathcal N\) is closed in
	\(\mathcal{AA}_{\rm orb}(X,T)\), since the inclusion into $\mathcal A(X,T)$ is continuous and
	$\mathcal N$ is closed in $\mathcal A(X,T)$.   Hence, as the quotient of a Banach space by a closed subspace,
	$\mathcal E_{\rm orb}$ is a Banach space.  	Let
	$
	\iota:\mathcal E_{\rm orb}\to\mathcal S
	$
	be the natural map.  Indeed, if $\iota([\Psi])=0$, then $\Psi\in\mathcal N$. Since
	$\Psi\in\mathcal{AA}_{\rm orb}(X,T)$, it follows that
	$\Psi\in\mathcal{AA}_{\rm orb}(X,T)\cap\mathcal N$, and hence
	$[\Psi]=0$ in $\mathcal E_{\rm orb}$. 
	
	For every $\eta=[\Psi]\in\mathcal E_{\rm orb}$ and 
	$\Phi\in\mathcal{AA}_{\rm orb}(X,T)\cap\mathcal N$, we have 
	$
	\|\iota(\eta)\|_{\mathcal S}
	\leq
	\|\Psi+\Phi\|_{\mathcal A}
	\leq
	\|\Psi+\Phi\|_{\mathcal{AA}_{\rm orb}}.
	$ 
	Taking the infimum over all such $\Phi$ gives
	$
	\|\iota(\eta)\|_{\mathcal S}
	\leq
	\|\eta\|_{\mathcal E_{\rm orb}}.
	$
	Thus $\iota$ is continuous. Moreover, Theorem~\ref{thm:orb-banach-dense}
	implies that
	$\mathcal{AA}_{\rm orb}(X,T)$ is dense in $\mathcal A(X,T)$.
    Finally,
	for \(\eta=[\Psi]\),
	$
	\sup_{\mu\in\M(X,T)}|\eta_*(\mu)|
	\leq
	\limsup_{n\to\infty}\frac1n\|\psi_n\|_\infty
	\leq
	\|\Psi\|_{\mathcal A}.
	$
	Taking the infimum over all representatives of $\eta$ in
	$\mathcal S$ gives
	$$
	\sup_{\mu\in\M(X,T)}|\eta_*(\mu)|
	\leq
	\|\eta\|_{\mathcal S}.
	$$
	The inequality
	$
	\|\eta\|_{\mathcal S}
	\leq
	\|\eta\|_{\mathcal E_{\rm orb}}
	$
	follows from the definitions of the quotient norms.
\end{proof}

We now fix an almost additive sequence $\Phi$ and introduce the
Lipschitz perturbation space $\cL_\Phi(X)$ on which the relative TPO theory will be formulated. 
\begin{definition} Let $(X,T)$ be a TDS and let
	$\Phi=\{\phi_n\}_{n\geq1}\subset C(X,\mathbb R)$
	be an almost additive sequence 
	and  \(f\in\Lip(X)\), define
	$$
	\Phi^f:=\Phi+S_\bullet f,
	\qquad \phi_n^f:=\phi_n+S_nf.
	$$
	The \emph{Lipschitz leaf over \(\Phi\)} is
	$$
	\cL_\Phi(X):=\{\Phi^f:f\in\Lip(X)\},
	$$
	equipped with the metric
	$$
	d_\Phi(\Phi^f,\Phi^g):=\|f-g\|_{\Lip}.
	$$
	Thus \(\cL_\Phi(X)\) is an affine Banach space canonically isometric to
	\(\Lip(X)\). We call \(f\) the perturbation function.
	The optimization  on the leaf is
	\begin{equation*}
		(\Phi^f)_*(\mu)=\Phi_*(\mu)+\int f\,d\mu.
	\end{equation*}
\end{definition}

$\mathcal L_\Phi(X)$ provides the natural
Lipschitz topology for perturbations of a fixed almost additive sequence. It should,
however, be distinguished from both the full Banach space of
asymptotically additive sequences and its asymptotic quotient. We now
introduce these two  objects and compare them with
$\mathcal L_\Phi(X)$.

Let \(q:\mathcal A(X,T)\to\mathcal S\) be the quotient map.  Cuneo's theorem
\cite{Cuneo2020} says that the bounded linear map
\begin{equation*}
	\pi:C(X,\mathbb{R})\longrightarrow\mathcal S,
	\qquad \pi(g):=[S_\bullet g],
\end{equation*}
is surjective.  It is therefore open by 
\cite[Lemma 2.4]{BomfimHuoVarandasZhao2023}.

\begin{theorem}
	\label{thm:leaf-quotient-dense}  Let $(X,T)$ be a TDS and let
	$\Phi=\{\phi_n\}_{n\geq1}\subset C(X,\mathbb R)$
	be an almost additive sequence. Then
	\begin{equation*}
		\overline{q(\cL_\Phi(X))}^{\,\|\cdot\|_{\mathcal S}}
		=\mathcal S.
	\end{equation*}
	In contrast, \(\cL_\Phi(X)\) is not dense in  
	\((\mathcal A(X,T),\|\cdot\|_{\mathcal A})\).
\end{theorem}

\begin{proof}
	Let \([\Psi]\in\mathcal S\).  By the surjectivity of \(\pi\), there exists
	\(g\in C(X,\mathbb{R})\) such that
	$$
	\pi(g)=[\Psi]-[\Phi].
	$$
	Choose \(f_k\in\Lip(X)\) with \(\|f_k-g\|_\infty\to0\).  Since
	$$
	\|\pi(f_k-g)\|_{\mathcal S}
	\leq \|f_k-g\|_\infty,
	$$
	we have
	$
	[\Phi+S_\bullet f_k]\longrightarrow[\Psi]
	\text{ in }\mathcal S.
	$
	This proves the density of $q(\cL_\Phi(X))$ in $\mathcal S$.
	
	To show that $\cL_\Phi(X)$ is not dense before passing to the quotient,
	consider the sequence $B=\{b_n\}_{n\geq1}$ of constant functions defined by
	$b_1\equiv0$, $b_2\equiv1$, $b_n\equiv0$ for every  $n\geq3.	$ 
	Since each $b_n$ takes values in $\{0,1\}$, we have
	$$
	\|b_{n+m}-b_n-b_m\circ T^n\|_\infty\leq2
	$$
	for all $n,m\geq1$. Thus $B$ is almost additive with
	$D(B)\leq2$. In particular, $\Phi+B\in\mathcal A(X,T)$.
	Fix $f\in\Lip(X)$ and set
	$
	\delta_f:=\|B-S_\bullet f\|_{\mathcal A}.
	$
	The terms corresponding to $n=1$ and $n=2$ give
	$
	\|f\|_\infty
	=
	\|b_1-S_1f\|_\infty
	\leq\delta_f
	$
	and
	$
	\|1-f-f\circ T\|_\infty
	=
	\|b_2-S_2f\|_\infty
	\leq2\delta_f.
	$
	Therefore,
	\begin{align*}
		1
		\leq
		\|1-f-f\circ T\|_\infty
		+\|f+f\circ T\|_\infty
		\leq
		2\delta_f+2\|f\|_\infty
		\leq4\delta_f.
	\end{align*}
	Hence
	$
	\|B-S_\bullet f\|_{\mathcal A}
	=\delta_f
	\geq\frac14
	$
	for every $f\in\Lip(X)$. Consequently,
	$$
	\operatorname{dist}_{\mathcal A}
	\bigl(\Phi+B,\cL_\Phi(X)\bigr)
	=
	\inf_{f\in\Lip(X)}
	\|B-S_\bullet f\|_{\mathcal A}
	\geq\frac14.
	$$
	Thus $\cL_\Phi(X)$ is not dense in
	$(\mathcal A(X,T),\|\cdot\|_{\mathcal A})$.
\end{proof}

\begin{remark}
	The above density  
	results naturally lead to the   question: Since $q(\mathcal L_\Phi(X))$ is dense in $\mathcal S$ for every
	almost additive sequence $\Phi$, does relative TPO on a fixed leaf 
	imply TPO in $\mathcal S$?
	In general, the answer is negative. Density  
	does not   transfer openness or Baire genericity between different
	topologies. Even the continuous open surjection
	$\pi:C(X,\mathbb{R})\longrightarrow\mathcal S$ does not, in general, guarantee that the image of an arbitrary residual set
	is residual. 
	
	For a fixed almost additive sequence $\Phi$, the leaf
	$\mathcal L_\Phi(X)$ is identified with $\Lip(X)$, and relative TPO is
	studied in the Banach topology of $\Lip(X)$, as in the classical setting.
	The openness, density and
	genericity in relative TPO are all understood in this topology. Since all
	perturbations $f$ remain in the same leaf, relative TPO does not depend on whether
	the leaf is dense in a larger sequence space.
	This also explains why the global Banach space $\mathcal E_{\rm orb}$ is
	needed. 
	
	In Theorem \ref{thm:leaf-quotient-dense}, we proved that the image of the leaf \(\cL_\Phi(X)\) is dense in 
	\(\mathcal S\). However, we cannot say that the  leaf is densely embedded into  \(\mathcal S\), because embedding typically requires the map to be injective, while  \(q|_{\cL_\Phi}\) need not be injective. For any $f,g \in Lip(X)$, 
	$$
	q(\Phi^{f})=q(\Phi^{g}) \iff S_\bullet(f-g)=\Phi^f-\Phi^g\in\mathcal N.
	$$
	For example, when
	\(h=u-u\circ T\in\Lip(X)\), \(u\in C(X,\mathbb{R})\),
	$
	\|S_nh\|_{\infty}/n=\|u-u\circ T^n\|_{\infty}/n\leq 2\|u\|_{\infty}/n.
	$
	Thus, $S_nh$ is sublinear in the norm $\|\cdot\|_{\mathcal{A}}$, so $S_\bullet h\in \mathcal{N}$. 
\end{remark}
Although $\cL_\Phi(X)$ is not dense in $\mathcal A(X,T)$, its closure
can be described explicitly.
\begin{proposition}
	\label{prop:raw-leaf-closure}
	Let $(X,T)$ be a TDS and let
	$\Phi=\{\phi_n\}_{n\geq1}\subset C(X,\mathbb{R})$ be an almost additive sequence.
	Then
	\begin{equation*}
		\overline{\cL_\Phi(X)}^{\,\|\cdot\|_{\mathcal A}}
		=\{\Phi+S_\bullet g:g\in C(X,\mathbb{R})\}.
	\end{equation*}
\end{proposition}

\begin{proof}
Fix $g\in C(X,\mathbb{R})$. By the uniform density of $\Lip(X)$ in $C(X,\mathbb{R})$,
there exists a sequence $\{f_k\}_{k\geq1}\subset\Lip(X)$ such that
$f_k\to g$ uniformly on $X$. Moreover,
$
\|(\Phi+S_\bullet f_k)-(\Phi+S_\bullet g)\|_{\mathcal A}
=
\|S_\bullet(f_k-g)\|_{\mathcal A}
\leq
\|f_k-g\|_\infty.
$
Therefore,
$
\{\Phi+S_\bullet g:g\in C(X,\mathbb{R})\}
\subset
\overline{\cL_\Phi(X)}^{\,\|\cdot\|_{\mathcal A}}.
$

Conversely, suppose that
$\Phi+S_\bullet f_k\to\Psi=\{\psi_n\}_{n\geq1}$
in $\|\cdot\|_{\mathcal A}$, where $f_k\in\Lip(X)$. The convergence
of the first coordinates gives
$
\|\phi_1+f_k-\psi_1\|_\infty\to0.
$
Set
$
g:=\psi_1-\phi_1.
$
Then $g\in C(X,\mathbb{R})$ and $f_k\to g$ uniformly on $X$.

For every fixed $n\geq1$, 
$
\|\phi_n+S_nf_k-\psi_n\|_\infty
\leq
n\|\Phi+S_\bullet f_k-\Psi\|_{\mathcal A}
\to0.
$
On the other hand,
$
\|S_nf_k-S_ng\|_\infty
\leq
n\|f_k-g\|_\infty
\to0.
$
It follows that $ psi_n=\phi_n+S_ng$
for every $n\geq1$. Hence $\Psi=\Phi+S_\bullet g$, which proves the reverse inclusion.
\end{proof}

\section{Maximizing sets for almost additive potentials}
This section studies maximizing sets for almost additive sequences. We first introduce maximizable and minimax sets. Next, we study completely maximizing sets and the subordination property, and give a boundedness condition under which a minimax set is completely maximizing and dynamically minimal. Finally, for orbit Lipschitz sequence, we use the subset closing property to verify this boundedness condition.

\begin{definition}
Let \(\Phi=\{\phi_n\}\subset C(X,\mathbb{R})\) be an almost additive.  A closed invariant set \(Y\subset X\) is
\emph{\(\Phi\)-maximizable} if
$$
 \M(Y,T)\cap\Max(X,T,\Phi)\neq\varnothing.
$$
A \(\Phi\)-maximizable set \(Y\) is \emph{minimax} if every
\(\mu\in\M(Y,T)\cap\Max(X,T,\Phi)\) satisfies \(\supp\mu=Y\).
\end{definition}

\begin{lemma}\label{lem:minimax-exists}
Let $(X,T)$ be a TDS and let
$\Phi=\{\phi_n\}_{n\geq1}\subset C(X,\mathbb{R})$ be an almost additive sequence.
Then there exists a minimax set for $\Phi$.
\end{lemma}
\begin{proof}
	Let $\mathscr Y_\Phi$ denote the collection of all $\Phi$-maximizable
	closed invariant subsets of $X$, ordered by set-theoretic inclusion.
	This collection is non empty because
	$
	\M(X,T)\cap\Max(X,T,\Phi)=\Max(X,T,\Phi)\neq\varnothing,
	$
	and hence $X\in\mathscr Y_\Phi$.
	
	We show that every totally ordered subfamily of $\mathscr Y_\Phi$ has a
	lower bound in $\mathscr Y_\Phi$. Let
	$\{Y_\alpha\}_{\alpha\in I}\subset\mathscr Y_\Phi$
	be totally ordered by inclusion, and set
	$
	Y_\infty:=\bigcap_{\alpha\in I}Y_\alpha.
	$
	Since every $Y_\alpha$ is closed and $T$-invariant, $Y_\infty$ is also
	closed and $T$-invariant.
	
	For each $\alpha\in I$, define
	$
	K_\alpha
	:=
	\M(Y_\alpha,T)\cap\Max(X,T,\Phi).
	$
	Because $Y_\alpha$ is $\Phi$-maximizable, $K_\alpha$ is non-empty.
	Moreover, $\M(Y_\alpha,T)$ is closed in $\M(X,T)$, while
	$\Max(X,T,\Phi)$ is compact. Hence $K_\alpha$ is compact.	
	If $Y_\alpha\subseteq Y_\beta$, then
	$
	\M(Y_\alpha,T)\subseteq\M(Y_\beta,T),
	$
	and therefore
	$
	K_\alpha\subseteq K_\beta.
	$
	Thus, the family $\{K_\alpha\}_{\alpha\in I}$ is totally ordered. In particular, it has the finite intersection property:
	Since all the sets
	$K_\alpha$ are compact subsets of the compact space
	$\Max(X,T,\Phi)$, it follows that
	$	\bigcap_{\alpha\in I}K_\alpha\neq\varnothing.	$	
	Choose
	$
	\mu\in\bigcap_{\alpha\in I}K_\alpha.
	$
	Then $\mu\in\Max(X,T,\Phi)$ and
	$
	\supp\mu\subseteq Y_\alpha$
	for every $\alpha\in I.
	$
	Consequently,
	$
	\supp\mu\subseteq\bigcap_{\alpha\in I}Y_\alpha=Y_\infty.
	$
	It follows that
	$
	\mu\in\M(Y_\infty,T)\cap\Max(X,T,\Phi),
	$
	so \(Y_\infty\) is \(\Phi\)-maximizable. Therefore,
	\(Y_\infty\in\mathscr Y_\Phi\), and it is a lower bound of 
	\(\{Y_\alpha\}_{\alpha\in I}\).
	
	By Zorn's lemma, 
	$(\mathscr Y_\Phi,\subseteq)$ has a minimal element $Y$. Thus, if
	$Y'\in\mathscr Y_\Phi$ and $Y'\subseteq Y$, then $Y'=Y$.	
	We claim that $Y$ is a minimax set for $\Phi$. Let
	$
	\mu\in\M(Y,T)\cap\Max(X,T,\Phi)
	$
	and set
	$
	Y':=\supp\mu.
	$
	The set $Y'$ is non-empty, closed and $T$-invariant. Moreover,
	$
	\mu\in\M(Y',T)\cap\Max(X,T,\Phi),
	$
	so $Y'$ is $\Phi$-maximizable, that is,
	$Y'\in\mathscr Y_\Phi$. Since $\mu$ is supported on $Y$, one has
	$Y'\subseteq Y$. The minimality of $Y$  gives
	$
	Y'=\supp\mu=Y.
	$
	 Hence $Y$ is a
	minimax set for $\Phi$.
\end{proof}

The following   estimate compares the normalized value of an almost
additive sequence along an orbit segment with the average of a fixed block
over its empirical measure.

\begin{lemma}\label{lem:blocking}
Let $\Phi=\{\phi_n\}_{n\geq1}\subset C(X,\mathbb{R})$ be an almost additive sequence. For
$x\in X$ and $n\geq1$, define
$$
 \nu_{x,n}:=\frac1n\sum_{j=0}^{n-1}\delta_{T^jx}.
$$
For any \(k\geq1\), with
\(M_k:=\max_{1\leq r\leq k}\|\phi_r\|_\infty\), one has, for \(n\geq k\),
\begin{equation}\label{eq:blocking}
 \left|
 \frac1n\phi_n(x)-\frac1k\int\phi_k\,d\nu_{x,n}
 \right|
 \leq \frac{3M_k}{n}+\frac{C_\Phi}{k}.
\end{equation}
\end{lemma}

\begin{proof}
For each  \(s\in\{0,\ldots,k-1\}\), we decompose the orbit segment into an
initial part of length \(s\), complete blocks of length \(k\), and a final part
of length less than \(k\).  Repeated use of \eqref{eq:aa} yields
$$
 \left|
 \sum_{j=0}^{q_s-1}\phi_k(T^{s+jk}x)-\phi_n(x)
 \right|
 \leq 2M_k+(q_s+1)C_\Phi,
$$
where \(q_s=\lfloor(n-s)/k\rfloor\).  Summing over \(s=0,\cdots,k-1\), we obtain 
$$
\left|\sum_{s=0}^{k-1}
\sum_{j=0}^{q_s-1}\phi_k(T^{s+jk}x)-k\phi_n(x)
\right|
\leq 2kM_k+C_\Phi\sum_{s=0}^{k-1}(q_s+1).
$$
 Since \(\sum_{s=0}^{k-1}q_s=n-k\), the full block
starting positions are precisely \(0,\ldots,n-k\).  Therefore,
$$
 \left|
 \sum_{i=0}^{n-k}\phi_k(T^ix)-k\phi_n(x)
 \right|
 \leq 2kM_k+nC_\Phi.
$$
Adding the last at most \(k-1\) terms,
$$
\left|
\sum_{i=0}^{n-1}\phi_k(T^ix)-k\phi_n(x)
\right|
\leq \left|
\sum_{i=0}^{n-k}\phi_k(T^ix)-k\phi_n(x)
\right|+(k-1)C_\Phi\leq 3kM_k+nC_{\Phi}.
$$
Dividing by \(kn\) gives
\eqref{eq:blocking}.
\end{proof}
The  above estimate allows us to relate long orbit segments to weak$^*$
limits of their empirical measures.  
We now use the relation to show that a
sufficiently long orbit segment in a minimax set cannot avoid a non-empty
relatively open subset while keeping its maximum ergodic average  at least
$\beta(X,T,\Phi)$.
\begin{lemma}\label{lem:avoidance}
Let $\Phi=\{\phi_n\}_{n\geq1}\subset C(X,\mathbb{R})$ be an almost additive
sequence, let $Y$ be a minimax set for $\Phi$, and let
$A\subset Y$ be a nonempty relatively open set.
 There exists \(N_A\geq1\) such that, if \(x\in Y\), \(n\geq N_A\),
and for $0\leq j<n$,
$
 T^jx\notin A,
$
then
\begin{equation}\label{eq:strict-avoidance}
 \frac1n\phi_n(x)<\beta(X,T,\Phi).
\end{equation}
\end{lemma}

\begin{proof}
Suppose that \eqref{eq:strict-avoidance} fails. Then there exist
sequences $n_i\to\infty$ and $x_i\in Y$ such that for $0\leq j<n_i$, $T^jx_i\notin A$
and
$$
\frac{\phi_{n_i}(x_i)}{n_i}\geq\beta(X,T,\Phi).
$$
Passing to a subsequence, assume
$\nu_{x_i,n_i}
=
\frac1{n_i}\sum_{j=0}^{n_i-1}\delta_{T^jx_i}.$  Since $Y$ is compact and invariant, then \(\nu\in\M(Y,T)\).  Since these measures as probability measures on $Y$, the portmanteau
theorem gives
$
\nu(A)\leq\liminf_{i\to\infty}\nu_{x_i,n_i}(A)=0.
$
Thus $\nu(A)=0$.

Fix $k\geq1$. Since $n_i\to\infty$, Lemma~\ref{lem:blocking} applies
for all sufficiently large $i$ and gives
\begin{align*}			 
	\frac{1}{k}\int\phi_{k}\,d\nu_{x_{i},n_{i}}
	\geq\frac{\phi_{n_{i}}(x_i)}{n_{i}}-\frac{3M_{k}}{n_{i}}-\frac{C_{\Phi}}{k}\geq\beta(X,T,\Phi)-\frac{3M_{k}}{n_{i}}-\frac{C_{\Phi}}{k}.
\end{align*}
Letting \(i\to\infty\), 
$\frac{1}{k}\int\phi_{k}\,d\nu
	\geq \beta(X,T,\Phi)-\frac{C_\Phi}{k}.$  
Combining this with \eqref{eq:uniform-approx},
$$
\Phi_*(\nu)\geq\frac{1}{k}\int\phi_k\,d\nu-\frac{C_{\Phi}}{k}\geq\beta(X,T,\Phi)-\frac{2C_{\Phi}}{k}. 
$$
Letting \(k\to\infty\)
shows that \(\Phi_*(\nu)\geq\beta(X,T,\Phi)\).  Thus \(\nu\) is maximizing.
But \(\supp\nu\subset Y\setminus A\), contradicting the minimax property of
\(Y\).
\end{proof}
Lemma~\ref{lem:avoidance} estimates all sufficiently long orbit segments avoiding $A$, while the  shorter ones are uniformly bounded, yielding the upper bound of  Minimax set.
\begin{proposition}\label{prop:minimax-bound}
Let $\Phi=\{\phi_n\}_{n\geq1}\subset C(X,\mathbb{R})$ be an almost additive sequence  such that \(\beta(X,T,\Phi)=0\) and \(Y\) is minimax for \(\Phi\).  For every
non-empty relatively open set \(A\subset Y\), there is a finite constant
\begin{equation*}
 K_A:=\max_{1\leq r<N_A}\|\phi_r\|_\infty
\end{equation*}
(with \(K_A=0\) if \(N_A=1\)) such that every orbit segment in \(Y\) avoiding
\(A\) satisfies
\begin{equation*}
 \phi_n(x)\leq K_A.
\end{equation*}
\end{proposition}

\begin{proof}
If \(n\geq N_A\), Lemma~\ref{lem:avoidance} gives \(\phi_n(x)<0\).  If
\(n<N_A\), by the definition of \(K_A\).
\end{proof}

For a closed non-empty invariant set \(Z\), write
$
 \beta(Z,T,\Phi)=\max_{\mu\in\M(Z,T)}\Phi_*(\mu),
 \Max(Z,T,\Phi)=\{\mu\in\M(Z,T):\Phi_*(\mu)=\beta(Z,T,\Phi)\}.
$

\begin{lemma}\label{lem:restriction}
Let $(X,T)$ be a TDS and $\Phi=\{\phi_n\}_{n\geq1}\subset C(X,\mathbb{R})$
an almost additive sequence. Suppose that $Z\subset X$ is non-empty,
closed, and $T$-invariant.  Then:
\begin{enumerate}[label=(\alph*)]
 \item \(\beta(Z,T,\Phi)\leq\beta(X,T,\Phi)\);
 \item \(Z\) is \(\Phi\)-maximizable if and only if
 \(\beta(Z,T,\Phi)=\beta(X,T,\Phi)\);
 \item if \(\beta(Z,T,\Phi)=\beta(X,T,\Phi)\), then
 \(\Max(Z,T,\Phi)\subset\Max(X,T,\Phi)\).
\end{enumerate}
\end{lemma}

\begin{proof}
The three statements follow immediately from
\(\M(Z,T)\subset\M(X,T)\) and the definitions.
\end{proof}

We next consider invariant sets on which every invariant measure is
maximizing. We also distinguish the stronger case in which these measures
are  all the maximizing measures.

\begin{definition}Let $(X,T)$ be a TDS and let
	$\Phi=\{\phi_n\}_{n\geq1}\subset C(X)$ be an almost additive sequence.  A closed non-empty invariant set \(Y\) is \emph{completely maximizing} for
\(\Phi\) if
$$
 \M(Y,T)\subset\Max(X,T,\Phi).
$$
It is \emph{subordination maximizing} if
$$
 \M(Y,T)=\Max(X,T,\Phi).
$$
\end{definition}

It follows immediately from the definition that every completely maximizing
set is maximizable. Moreover, every nonempty closed invariant subset of a
completely maximizing set is itself completely maximizing. 

The following lemma shows that a minimax set is dynamically minimal whenever
it is completely maximizing.
\begin{lemma}\label{lem:minimal}
Let $(X,T)$ be a TDS and let $\Phi=\{\phi_n\}_{n\geq1}\subset C(X,\mathbb{R})$ be an almost additive sequence. If \(Y\) is both minimax and completely maximizing for \(\Phi\), then
\((Y,T)\) is dynamically minimal.
\end{lemma}

\begin{proof}
Let \(Y'\subset Y\) be a non-empty, closed and invariant set.  Any invariant measure
supported on \(Y'\) is maximizing because \(Y\) is completely maximizing.  Its
support must equal \(Y\) because \(Y\) is minimax.  Hence \(Y'=Y\).
\end{proof}

We next extend the bounded-Birkhoff-sum criterion of Morris \cite{Morris2007}.
The proof is included because the almost-additivity errors have to be controlled
at every concatenation.

\begin{definition}
The potential \(\Phi\subset C(X,\mathbb{R})\) has the \emph{subordination property} if
$$
 \eta\in\Max(X,T,\Phi),\quad
 \nu\in\M(X,T),\quad
 \supp\nu\subset\supp\eta
 \quad\Longrightarrow\quad
 \nu\in\Max(X,T,\Phi).
$$
\end{definition}

\begin{theorem}\label{thm:aa-morris} Let $(X,T)$ be a TDS and let $\Phi=\{\phi_n\}_{n\geq1}\subset C(X,\mathbb{R})$ be an almost additive sequence with 
  \(\beta(X,T,\Phi)=0\) and
\begin{equation}\label{eq:bounded-aa}
 B_\Phi:=\sup_{n\geq1}\sup_{x\in X}\phi_n(x)<\infty.
\end{equation}
Then \(\Phi\) has the subordination property.  Consequently, it admits a
subordination maximizing set.
\end{theorem}

\begin{proof}
We prove the contrapositive form of subordination.  Let
\(\nu,\mu\in\M(X,T)\) satisfy \(\supp\nu\subset\supp\mu\) and
\(\Phi_*(\nu)<0\).  By lemma \ref{lem:functional}, \(\Phi_*\) is affine. Together with the ergodic decomposition theorem,  there are a probability space \((\Omega,\mathcal{F},\mathbb{P})\) and a collection of \(T\)-invariant ergodic Borel probability measure \(\{\mu_{\omega}:\omega\in\Omega\}\) on \(X\) such that 
$$
\Phi_*(\nu)=\int\Phi_*(\nu_{\omega})\,d\mathbb{P}(\omega).
$$
There exist some ergodic measure
\(\nu'\) with \(\Phi_*(\nu')<0\), and
\(\supp\nu'\subset\supp\nu\).

Choose
$$
 D>\max\{B_\Phi,0\}+2C_\Phi+1.
$$
By \eqref{eq:aa-ergodic}, there exist \(x_0\in\supp\nu'\) and \(N\geq1\) such
that \(\phi_N(x_0)<-4D\).  Since $\phi_N\in C(X,\mathbb{R})$, there is an open neighborhood
\(U\ni x_0\) on which \(\phi_N(x)<-3D\) for any \(x\in U\).

Since $\supp\nu\subset\supp\mu$, $\mu(U)>0$. By the ergodic decomposition theorem, 
$$
\mu(U)=\int_{\Omega}\mu_{\omega}(U)\,d\mathbb{P}(\omega).
$$
Let \(\mu_\omega\) be an ergodic measure with
\(\mu_\omega(U)>0\).  Define the return times of $x$ to $U$ by
\(
r_0(x):=0, r_n(x):=
\min\bigl\{k>r_{n-1}(x):T^kx\in U\bigr\}\) for 
\(n\geq1\). 
By the almost additive ergodic theorem and Birkhoff's pointwise
ergodic theorem applied to $\mathbf{1}_U$, we may choose
\(x_0\in U\) such that
$$
\lim_{m\to\infty}\frac{\phi_m(x_0)}{m}
=\Phi_*(\mu_\omega),\qquad
\lim_{m\to\infty}
\frac{1}{m}\sum_{k=0}^{m-1}\mathbf{1}_U(T^kx_0)
=\mu_\omega(U).
$$
Consequently, \(n/r_n(x_0)\to \mu_\omega(U)\). Since \(\mu_{\omega}(U)>0\), 
among the first \(n\)
returns to \(U\), we can choose at least  \(M_{n}=\lfloor n/N\rfloor\) return times separated by at least \(N\), forming an increasing sequence \(\{{n_i}\}_{i=1}^{M_n}\) with \(n_1=0,n_{M_n}=r_{n}(x_{\omega})\), and such that for each \(i\), \(T^{n_{i}}x_{\omega}\in U\) and \(n_{i+1}\geq n_{i}+N\).  Thus, for each \(n\), 
\begin{align*}
	\phi_{r_n}(x)&\leq\phi_{n_{M_{n}}-(n_{M_{n}}-1)}(T^{n_{M_{n}}-1}x)+\phi_{n_{M_{n}-1}}(x)+C_{\Phi}\\ 
	&\leq\phi_{N}(T^{n_{M_{n}}-1}x)+\phi_{n_{M_{n}}-(n_{M_{n}}-1)-N}(T^{n_{M_{n}}-1+N}x)+\phi_{n_{M_{n}}-1}(x)+2C_{\Phi}\\
	&\leq -(3D-B_{\Phi}-2C_{\Phi})+\phi_{n_{M_{n}}-1}(x)\\
	&\leq -(3D-B_\Phi-2C_\Phi)M_n<-DM_n.
\end{align*}
Consequently, 
\begin{align}\label{eq:component-negative}
 \Phi_*(\mu_\omega)=\lim_{n\to\infty}\frac{1}{r_n(x)}\phi_{r_{n}}(x)\leq\lim_{n\to\infty}-\frac{DM_n}{r_n(x)}\leq -\frac{D}{N}\mu_\omega(U)<0.
\end{align}
For components with \(\mu_\omega(U)=0\), \(\beta(X,T,\Phi)=0\) gives
\(\Phi_*(\mu_\omega)\leq0\).  Since \(x_0\in\supp\mu\), one has
\(\mu(U)>0\).  Integrating \eqref{eq:component-negative} over the ergodic
decomposition gives
\begin{align*}
	\Phi_*(\mu)=\int_{\mu_{\omega}(U)>0}\Phi_*(\mu_{\omega})\,d\mathbb{P}(\omega)+\int_{\mu_{\omega}(U)=0}\Phi_*(\mu_{\omega})\,d\mathbb{P}(\omega)\leq-\frac{D}{N}\mu(U)<0.
\end{align*}
Hence \(\mu\) is not maximizing.  This is the contrapositive of the
subordination property.

It remains to construct a subordination maximizing set. We first choose a dense sequence
\(\{\eta_j\}\) in the compact metrizable set \(\Max(X,T,\Phi)\), set
\(\eta=\sum_{j\geq1}2^{-j}\eta_j\), and let \(K=\supp\eta\). Since \(\eta_j\in\Max(X,T,\phi)\) and \(\Max(X,T,\phi)\) is closed and convex. Thus, 
\(\eta\in\Max(X,T,\phi)\).

If \(\rho\in\M(X,T)\) is supported on \(K\), the
subordination property ensures \(\rho\) is \(\Phi\)-maximizing.  Conversely, every maximizing
measure is a weak* limit of a sequence chosen from the dense set
\(\{\eta_j:j\geq1\}\).  Since every \(\eta_j\) is supported on the closed set
\(K\), the limit is supported on \(K\) as well.  Thus
\(\M(K,T)=\Max(X,T,\Phi)\).
\end{proof}
We next give a criterion for a set to be completely maximizing.
\begin{proposition}\label{prop:complete}
Let $(X,T)$ be a TDS and let $\Phi=\{\phi_n\}_{n\geq1}\subset C(X,\mathbb{R})$ be an almost additive sequence such that
\(\beta(X,T,\Phi)=0\).  If \(Y\) is minimax for \(\Phi\) and
\begin{equation*}
 \sup_{n\geq1}\sup_{x\in Y}\phi_n(x)<\infty,
\end{equation*}
then \(Y\) is completely maximizing and dynamically minimal.
\end{proposition}

\begin{proof}
Apply Theorem~\ref{thm:aa-morris} to  \((Y,T)\).  It gives a
closed invariant \(Z\subset Y\) with
\( \M(Z,T)=\Max(Y,T,\Phi).\)
Since \(Y\) is minimax for \((X,T,\Phi)\), and hence is \((X,T,\Phi)\) maximizable, Lemma~\ref{lem:restriction} gives
\(\Max(Y,T,\Phi)\subset\Max(X,T,\Phi)\). Each \(\mu\in\M(Z,T)\)
is therefore maximizing on \(X\), and the minimax property of \(Y\) forces  \(\supp(\mu)=Y\).  Hence \(Z=Y\), so
\(\M(Y,T)\subset\Max(X,T,\Phi)\).  Since \(Y\) is minimax and completely maximizing, by Lemma~\ref{lem:minimal}, \(Y\) is dynamically minimal. 
\end{proof}

The avoidance lemma controls orbit segments that do not visit a  
open subset of a minimax set. We now introduce a closing condition under which a minimax set for an orbit
Lipschitz almost additive sequence is completely maximizing and dynamically
minimal. 
\begin{definition}
	A non-empty subset \(A\subset X\) has the \emph{closing property relative to
		\((X,T)\)} if there exists \(C_A>0\) such that, whenever \(x\in A\) and
	\(T^nx\in A\) for some \(n\geq1\), there exists \(z=T^nz\) satisfying
	\begin{equation*}
		\sum_{j=0}^{n-1}d(T^jx,T^jz)\leq C_A.
	\end{equation*}
	A closed invariant set \(Y\subset X\) has the \emph{subset closing property}
	if some non empty relatively open set \(A\subset Y\) has the closing property
	relative to \((X,T)\).
\end{definition}
\begin{theorem}\label{thm:subset-closing}
	Let $(X,T)$ be a TDS, and let
	$\Phi=\{\phi_n\}_{n\geq1}\in
	\mathcal{AA}_{\rm orb}(X,T)$.
	If $Y\subset X$ is a minimax set for $\Phi$ and $Y$ has the
	subset closing property, then $Y$ is completely maximizing for
	$\Phi$ and dynamically minimal.
\end{theorem}

\begin{proof}
	Without loss of generality, assume \(\beta(X,T,\Phi)=0\).  By
	Proposition~\ref{prop:complete}, it suffices to prove
	\begin{equation}\label{eq:goal-bound}
		\sup_{n\geq1}\sup_{x\in Y}\phi_n(x)<\infty.
	\end{equation}
	Choose a non empty relatively open set \(A\subset Y\) with closing constant
	\(C_A\).  Let \(N_A\) and \(K_A\) be given by
	Proposition~\ref{prop:minimax-bound}.  If the orbit segment
	\(x,Tx,\ldots,T^{n-1}x\notin A\), then \(\phi_n(x)\leq K_A\).
	
	Now assume that $T^jx\in A$ for some $0\leq j<n$. Set
	$
	n_1:=\min\{0\leq j<n:T^jx\in A\},$ 
	$n_2:=\max\{0\leq j<n:T^jx\in A\},$  hence $\phi_{n_1}(x)\leq K_{A}$. 
	If \(p:=n_2-n_1\geq1\), put
	\(y=T^{n_1}x\).  The closing property gives a point \(z=T^pz\) such that
	$
	\sum_{j=0}^{p-1}d(T^jy,T^jz)\leq C_A.
	$
	Let $\mu_z$ be the periodic measure supported on the orbit of $z$. Since
	$\mu_z\in\M(X,T)$ and $\beta(X,T,\Phi)=0$, we have
	$
	\Phi_*(\mu_z)\leq\beta(X,T,\Phi)=0.
	$  Repeated
	use of the  almost additive inequality gives
	$
	\phi_{kp}(z)\geq k\phi_p(z)-(k-1)C_\Phi.
	$
	After division by \(kp\) and take $k\to\infty$,
	$
	0\geq\Phi_*(\mu_z)\geq(\phi_p(z)-C_\Phi)/p,
	$
	so \(\phi_p(z)\leq C_\Phi\).  The orbit Lipschitz condition now gives
	\begin{equation}\label{eq:middle-bound}
		\phi_p(y)\leq C_\Phi+L_{\rm orb}(\Phi)C_A.
	\end{equation}
	Since for $n_2<j<n,$ $
	T^jx\notin A$, hence \begin{equation}\label{eq:tail-bound}
		\phi_{n-n_2}(T^{n_2}x)
		\leq
		\|\phi_1\|_\infty+K_A+C_\Phi.
	\end{equation}
	Finally,  combining
	\eqref{eq:middle-bound} and \eqref{eq:tail-bound} gives  
	\begin{equation*}
		\phi_n(x)
		\leq 2K_A+\|\phi_1\|_\infty+L_{\rm orb}(\Phi)C_A+4C_\Phi.
	\end{equation*}
	If $n_1=n_2$, the above estimate still holds. Therefore,
	\eqref{eq:goal-bound} holds. The result follows from
	Proposition~\ref{prop:complete}.
\end{proof}

\section{Global  typical periodic optimization}
This section studies TPO in the full space $\cE_{\rm orb}$. We introduce global
maximizable families and prove the Baire and subsystem reduction results.
We then establish periodic orbit domination and global locking, which lead
to equivalent forms of global TPO and the global structural theorem.
\begin{definition}
	Let $(X,T)$ be a TDS, and let $\cE_{\rm orb}=\cE_{\rm orb}(X,T)$ be the quotient Banach space of orbit Lipschitz almost additive sequences. For $\eta\in\cE_{\rm orb}$, write
	$
	\beta(\eta):=\sup_{\mu\in\M(X,T)}\eta_*(\mu)
	$
	and
	$\Max(\eta)	:=	\{\mu\in\M(X,T):\eta_*(\mu)=\beta(\eta)\}.$ 
	For any $\mathscr N\subset\M(X,T)$, define
	\begin{align*}
		\cE_{\mathscr N}
		:=\{\eta\in\cE_{\rm orb}:
		\Max(\eta)\cap\mathscr N\neq\varnothing\},\qquad
		\cE_{\mathscr N}^{\subset}
		:=\{\eta\in\cE_{\rm orb}:
		\Max(\eta)\subset\mathscr N\}.
	\end{align*}
	If $Z\subset X$ is a non-empty closed $T$-invariant set, write
	$
	\cE_Z:=\cE_{\M(Z,T)}.$ 
	Hence, $\eta\in\cE_Z$ if and only if $Z$ supports at least one $\eta$-maximizing measure.
\end{definition}

\begin{lemma}\label{g:lem:closedness-global}
	Let $(X,T)$ be a TDS. If $\mathscr N$ is closed in $\M(X,T)$, then
	$\cE_{\mathscr N}$ is closed in $\cE_{\rm orb}$.  In particular,
	$\cE_Z$ is closed for every closed invariant set $Z$.
\end{lemma}
\begin{proof}
	Let $\eta_k\to\eta$ in $\cE_{\rm orb}$. For each $k\geq1$, choose
	$
	\mu_k\in\Max(\eta_k)\cap\mathscr N.
	$
	Since $\M(X,T)$ is weak$^*$ compact, we may pass to a subsequence.
	Without loss of generality, assume that
	$
	\mu_k\longrightarrow\mu
	$
	for some $\mu\in\M(X,T)$. Since $\mathscr N$ is   closed, we have $\mu\in\mathscr N$.
	
	Since
	$$
	\sup_{\rho\in\M(X,T)}
	|(\eta_k)_*(\rho)-\eta_*(\rho)|
	\leq
	\|\eta_k-\eta\|_{\cE_{\rm orb}}
	\to0.
	$$
	Moreover, the map
	$
	\rho\longmapsto\eta_*(\rho)
	$
	is continuous in the weak$^*$ topology. Therefore,
	$$\left|(\eta_k)_*(\mu_k)-\eta_*(\mu)\right|
	\leq
	\left|(\eta_k)_*(\mu_k)-\eta_*(\mu_k)\right|+\left|\eta_*(\mu_k)-\eta_*(\mu)\right|
	\to0.$$
	Thus,
	$
	(\eta_k)_*(\mu_k)\to\eta_*(\mu).
	$
	
	For each $\nu\in\M(X,T)$, the maximality of $\mu_k$ gives
	$
	(\eta_k)_*(\mu_k)\geq(\eta_k)_*(\nu).
	$
	Taking $k\to\infty$, we  obtain
	$
	\eta_*(\mu)\geq\eta_*(\nu)
	$
	for every $\nu\in\M(X,T)$. Hence
	$
	\mu\in\Max(\eta).
	$
	Therefore, 
	$\cE_{\mathscr N}$ is closed in $\cE_{\rm orb}$.
\end{proof}

We now introduce the global  maximizable family.   A family of closed invariant sets is called globally maximizable
if every class $\eta\in\cE_{\rm orb}$ has a maximizing measure supported on
at least one set in the family.  Thus, the corresponding sets $\cE_Z$ cover
the whole space $\cE_{\rm orb}$.  For a countable family, this covering and
the Baire category theorem give the open dense set in
Theorem~\ref{g:thm:Baire-global}.
\begin{definition}
	A collection $\gZ$ of closed invariant subsets of $X$ is
	\emph{$\cE_{\rm orb}$-maximizable} if, for every
	$\eta\in\cE_{\rm orb}$, some $Z\in\gZ$ is maximizable for
	$\eta$.  Equivalently,
	\begin{equation*}
		\cE_{\rm orb}=\bigcup_{Z\in\gZ}\cE_Z.
	\end{equation*}
\end{definition}

\begin{theorem}\label{g:thm:Baire-global}
	If $\gZ=\{Z_i:i\geq1\}$ is a countable
	$\cE_{\rm orb}$-maximizable family, then
	\begin{equation*}
		\bigcup_{i\geq1}\interior(\cE_{Z_i})
	\end{equation*}
	is open and dense in $\cE_{\rm orb}$.
\end{theorem}

\begin{proof}
	By the definition of an $\cE_{\rm orb}$ maximizable family,
	$\cE_{\rm orb}=\bigcup_{i\geq1}\cE_{Z_i}.$  
	By Lemma~\ref{g:lem:closedness-global}, each $\cE_{Z_i}$ is closed in
	$\cE_{\rm orb}$. Since $\cE_{\rm orb}$ is a Banach space, it is a
	Baire space.
	
	Let $U\subset\cE_{\rm orb}$ be a non-empty open set. Then $	U=\bigcup_{i\geq1}\bigl(U\cap\cE_{Z_i}\bigr),$ where each $U\cap\cE_{Z_i}$ is closed in $U$. By the Baire category
	theorem, there exists $i\geq1$ such that $U\cap\cE_{Z_i}$ has
	non-empty interior relative to $U$. Since $U$ is open in
	$\cE_{\rm orb}$, it follows that
	$U\cap\interior(\cE_{Z_i})\neq\varnothing.$  
	Thus every non-empty open subset of $\cE_{\rm orb}$ meets
	$\bigcup_{i\geq1}\interior(\cE_{Z_i})$, so this union is dense.
	It is open because it is a union of open sets.
\end{proof}

We next compare the two global optimization classes
$\cE_{\mathscr N}$ and $\cE_{\mathscr N}^{\subset}$. The following
lemma shows that they have the same interior when $\mathscr N$ is
closed and convex. Thus, on an open set, requiring at least one
maximizing measure to belong to $\mathscr N$ is equivalent to
requiring all maximizing measures to belong to $\mathscr N$.

\begin{lemma}\label{g:lem:common-interior-global}
	Let $(X,T)$ be a TDS. If $\mathscr N\subset\M(X,T)$ is closed and convex, then
	$$\interior(\cE_{\mathscr N})
	=\interior(\cE_{\mathscr N}^{\subset}).	$$
\end{lemma}

\begin{proof}
	It suffices to prove
	$
	\interior(\cE_{\mathscr N})
	\subset
	\interior(\cE_{\mathscr N}^{\subset}).
	$
	If $\mathscr N=\varnothing$, the conclusion is immediate. Thus, assume
	that $\mathscr N\neq\varnothing$. 
	
	Suppose that
	$\eta\in\interior(\cE_{\mathscr N})$ and let
	$\mu\in\Max(\eta)\setminus\mathscr N$. By the strong separation theorem
	\cite[Corollary~5.59]{AliprantisBorder1999}, there exists
	$f\in C(X,\mathbb R)$ such that for any 
	$\delta>0$, $\int f_0\,d\mu\geq3\delta$,  and
	$
	\sup_{\nu\in\mathscr N}\int f_0\,d\nu\leq0.
	$
	Since $\Lip(X)$ is dense in $C(X,\mathbb{R})$, choose $f\in\Lip(X)$ such that  $\|f-f_0\|_{\infty}<\delta$. Then $\int f\,d\mu\geq2\delta$, $\sup_{\nu\in\mathscr N}\int f\,d\nu\leq\delta$.  The additive class $[S_\bullet f]$ belongs to
	$\cE_{\rm orb}$.  For every $t>0$ and every $\nu\in\mathscr N$,
	$(\eta+t[S_\bullet f])_*(\mu)
	-(\eta+t[S_\bullet f])_*(\nu)
	\geq t\left(\int f\,d\mu-\int f\,d\nu\right)
	\geq t\delta.$
	Hence, $ \Max\bigl(\eta+t[S_\bullet f]\bigr)\cap\mathscr N
	=\varnothing.$
	Since
	$t[S_\bullet f]\to0$ in $\cE_{\rm orb}$, this contradicts
	$\eta\in\interior(\cE_{\mathscr N})$.
\end{proof}

Theorem \ref{g:thm:Baire-global} gives an open dense set where some $Z_i$ supports a maximizing measure. By the lemma \ref{g:lem:common-interior-global}, on the same open dense set, all maximizing measures are supported on $Z_i$. This gives the following corollary.

\begin{corollary}
	Under the hypotheses of Theorem~\ref{g:thm:Baire-global}, the set $$ \bigcup_{i\geq1}\interior(\cE_{Z_i}) = \bigcup_{i\geq1} \interior\bigl(\cE_{\M(Z_i,T)}^{\subset}\bigr) $$ is open and dense in $\cE_{\rm orb}$.
\end{corollary}

\begin{proof}
	Since $\M(Z_i,T)$ is closed and convex, Lemma~\ref{g:lem:common-interior-global} gives the equality.
\end{proof}

\begin{remark}
	Theorem~\ref{g:thm:Baire-global} is the global orbit Lipschitz version of the
	countable maximizing family.  Combined with the locking
	theorem, it supplies the topological part of the TPO framework.  To deduce a
	full structural theorem from TPO hypotheses on the subsystems $Z_i$, one must
	also transfer small perturbations from $Z_i$ to $X$.  
\end{remark}
If $Z\subset X$ is non-empty, closed and invariant, restriction to $Z$ defines a linear map 
\begin{equation*}
	R_Z:\cE_{\rm orb}(X,T)\to \cE_{\rm orb}(Z,T),
	\qquad R_Z[\Phi]=[\Phi|_Z].
\end{equation*}

\begin{lemma}\label{x:lem:contractive-restriction}
	The map $R_Z$   is well defined and contractive.
	Its kernel is closed.
\end{lemma}
\begin{proof}
	Let $\Phi=\{\phi_n\}_{n\geq1}\in\mathcal{AA}_{\rm orb}(X,T)$.  We first show that $\Phi|_Z\in\mathcal{AA}_{\rm orb}(Z,T)$.
	Since $\Phi\in\mathcal A(X,T)$, for any $\varepsilon>0$, there exists
	$g_\varepsilon\in C(X,\mathbb{R})$ such that
	$
	\limsup_{n\to\infty}
	\frac1n\|\phi_n-S_ng_\varepsilon\|_{\infty,X}
	\leq\varepsilon.
	$
	Since $Z$ is invariant,
	$
	S_n(g_\varepsilon|_Z)=(S_ng_\varepsilon)|_Z.$
	Therefore, $$\limsup_{n\to\infty}
	\frac1n
	\|\phi_n|_Z-S_n(g_\varepsilon|_Z)\|_{\infty,Z}
	\leq
	\limsup_{n\to\infty}
	\frac1n
	\|\phi_n-S_ng_\varepsilon\|_{\infty,X}
	\leq\varepsilon.$$
	Hence $\Phi|_Z\in\mathcal A(Z,T)$.
	
	Moreover, $$
	\|\Phi|_Z\|_{\mathcal A}
	=\sup_{n\geq1}\frac1n\|\phi_n|_Z\|_{\infty,Z}
	\leq
	\sup_{n\geq1}\frac1n\|\phi_n\|_{\infty,X}
	=
	\|\Phi\|_{\mathcal A}.$$
	For every $n,m\geq1$, we   have
	$\|\phi_{n+m}|_Z-\phi_n|_Z		-(\phi_m|_Z)\circ T^n\|_{\infty,Z}\leq\|	\phi_{n+m}-\phi_n-\phi_m\circ T^n
	\|_{\infty,X}.$  
	Taking the supremum over $n,m\geq1$ gives
	$
	D(\Phi|_Z)\leq D(\Phi).
	$ 
	For $x,y\in Z$ and $n\geq1$,  
	$$
	\frac{|\phi_n(x)-\phi_n(y)|}{
	\sum_{j=0}^{n-1}d(T^jx,T^jy)}
	\leq L_{\rm orb}(\Phi).
	$$ 
	Taking the supremum over $n\geq1$ and  $x\neq y\in Z$, we have
	$
	L_{\rm orb}(\Phi|_Z)\leq L_{\rm orb}(\Phi).
	$
	It follows that
	$
	\|\Phi|_Z\|_{\mathcal{AA}_{\rm orb}}
	\leq
	\|\Phi\|_{\mathcal{AA}_{\rm orb}}.
	$
	
	We next show that $R_Z$ is well defined on $\cE_{X}$.  Suppose
	that $[\Phi]=[\Psi]\in\cE_{\rm orb}(X,T)$, where
	$\Psi=\{\psi_n\}_{n\geq1}$.  Then $\Phi-\Psi$ is sublinear uniformly on
	$X$, and hence $\limsup_{n\to\infty}
	\frac1n
	\|(\phi_n-\psi_n)|_Z\|_{\infty,Z}
	\leq
	\limsup_{n\to\infty}
	\frac1n
	\|\phi_n-\psi_n\|_{\infty,X}=0.$ 
	Therefore,
	$
	[\Phi|_Z]=[\Psi|_Z]\in\cE_{\rm orb}(Z,T).
	$  Hence it is well defined.  Its linearity follows
	directly from the definition.

	Let $\eta=[\Phi]\in\cE_{\rm orb}(X,T)$, and let $ \Gamma=\{\gamma_n\}_{n\geq1} \in\mathcal{AA}_{\rm orb}(X,T)\cap\mathcal N. $ Then, $\Gamma|_Z\in\mathcal{AA}_{\rm orb}(Z,T)$. Moreover, $ \limsup_{n\to\infty} \frac1n\|\gamma_n|_Z\|_{\infty,Z} \leq \limsup_{n\to\infty} \frac1n\|\gamma_n\|_{\infty,X} =0. $  Hence $ [(\Phi+\Gamma)|_Z]=[\Phi|_Z]=R_Z\eta. $ By the definition of the quotient norm, $$\|R_Z\eta\|_{\cE_{\rm orb}(Z,T)} \leq \|(\Phi+\Gamma)|_Z\|_{\mathcal{AA}_{\rm orb}(Z,T)}\leq \|\Phi+\Gamma\|_{\mathcal{AA}_{\rm orb}(X,T)}.$$ Taking the infimum over all $\Gamma\in\mathcal{AA}_{\rm orb}(X,T)\cap\mathcal N$, we have $ \|R_Z\eta\|_{\cE_{\rm orb}(Z,T)} \leq \|\eta\|_{\cE_{\rm orb}(X,T)}. $ Therefore, $R_Z$ is contractive  and therefore continuous.

	Finally, let $\eta_k\in\ker R_Z$ and suppose that
	$\eta_k\to\eta$ in $\cE_{\rm orb}(X,T)$.  By the continuity of $R_Z$,
	$R_Z\eta	=\lim_{k\to\infty}R_Z\eta_k
	=0.$ 
	Thus $\eta\in\ker R_Z$, and hence $\ker R_Z$ is closed.
\end{proof}
The  subsystem space is the Banach quotient $\cE_{\rm orb}^{X}(Z,T)
:=\cE_{\rm orb}(X,T)/\ker R_Z.$ 
We identify this quotient algebraically with $R_Z(\cE_{\rm orb}(X,T))$ and equip it with
the lifting norm
\begin{equation}\label{x:eq:lifting-norm}
	\|\zeta\|_{Z\leftarrow X}
	:=\inf\{\|\theta\|_{\cE_{\rm orb}(X,T)}:R_Z\theta=\zeta\}.
\end{equation}
Thus $\cE_{\rm orb}^{X}(Z,T)$ consists precisely of orbit Lipschitz almost
additive classes on $Z$ which admit a global lift. The lifting norm is used to control perturbations when they are lifted from
$Z$ to $X$. The following proposition shows that the extendable space is
complete and that every small class on $Z$ has a small global lift.

\begin{proposition}\label{x:prop:quotient-lifting}
	The space
	$\bigl(\cE_{\rm orb}^{X}(Z,T),\|\cdot\|_{Z\leftarrow X}\bigr)$ is Banach.
	The map 
	$
	R_Z^X:\cE_{\rm orb}(X,T)\to\cE_{\rm orb}^{X}(Z,T)
	$ 
	is a continuous open surjection, and
	\begin{equation}\label{x:eq:two-restriction-norms}
		\|\zeta\|_{\cE_{\rm orb}(Z,T)}
		\leq \|\zeta\|_{Z\leftarrow X}.
	\end{equation}
	In particular, whenever
	$\|\zeta\|_{Z\leftarrow X}<\varepsilon$, there is a lift $\theta$ satisfying
	$R_Z\theta=\zeta$ and $\|\theta\|_{\cE_{\rm orb}(X,T)}<\varepsilon$.
\end{proposition}

\begin{proof}
	The quotient of the Banach space $\cE_{\rm orb}(X,T)$ by the closed subspace
	$\ker R_Z$ is Banach.  The quotient map is open, and its quotient norm is
	exactly \eqref{x:eq:lifting-norm}.  Inequality
	\eqref{x:eq:two-restriction-norms} follows from the contractivity in
	Lemma~\ref{x:lem:contractive-restriction}.  
\end{proof}

For $\zeta\in\cE_{\rm orb}^{X}(Z,T)$, write
 \begin{align*}
 	\beta_Z(\zeta)
 	:&=\max_{\mu\in\M(Z,T)}\zeta_*(\mu),\\
 	\mathcal M_{\max,Z}(\zeta)
 	:&=\{\mu\in\M(Z,T):\zeta_*(\mu)=\beta_Z(\zeta)\}
 \end{align*} 
If $\mathscr N\subset\M(X,T)$, put
\begin{equation*}
	\mathcal W_Z^X(\mathscr N)
	:=\{\zeta\in\cE_{\rm orb}^{X}(Z,T):
	\mathcal M_{\max,Z}(\zeta)\cap\mathscr N\cap\M(Z,T)\neq\varnothing\}.
\end{equation*}
With the preceding notation, we obtain the following global subsystem
reduction theorem.
\begin{theorem}
	\label{x:thm:subsystem-reduction-extendable}
	Let $U\subset\cE_{\rm orb}(X,T)$ be non-empty and open, and suppose that $Z$ is $\eta$-maximizable for every $\eta\in U$.  
	If $\mathcal W_Z^X(\mathscr N)$ is dense in
	$\cE_{\rm orb}^{X}(Z,T)$, then
	$\cE_{\mathscr N}\cap U$is dense in $U.$
\end{theorem}

\begin{proof}
	Fix $\eta\in U$ and $\varepsilon>0$ so small that
	$B_{\cE_{\rm orb}(X,T)}(\eta,\varepsilon)\subset U$.  By density, choose
	$\zeta\in\mathcal W_Z^X(\mathscr N)$ with
	$
	\|\zeta-R_Z\eta\|_{Z\leftarrow X}<\varepsilon.
	$
	Proposition~\ref{x:prop:quotient-lifting} provides $\theta\in\cE_{\rm orb}(X,T)$ such that
	$
	R_Z\theta=\zeta-R_Z\eta$,
 	and $\|\theta\|_{\cE_{\rm orb}(X,T)}<\varepsilon.
	$ 
	Hence $\eta+\theta\in U$ and $R_Z(\eta+\theta)=\zeta$.  Choose
	$
	\mu\in\mathcal M_{\max,Z}(\zeta)\cap\mathscr N\cap\M(Z,T).
	$
	$Z$ is maximizable for $\eta+\theta$.  Therefore, we have 
	$
	\beta_Z(R_Z(\eta+\theta))=\beta(\eta+\theta),$ and $\mathcal M_{\max,Z}(R_Z(\eta+\theta))\subset\Max(\eta+\theta).$ 
	It follows that $\mu\in\Max(\eta+\theta)\cap\mathscr N$, so
	$\eta+\theta\in\cE_{\mathscr N}\cap U$.  
	Therefore, $\cE_{\mathscr N}\cap U$ is dense in $U$.	
\end{proof}
\begin{remark}[Difference from the classical subsystem reduction]
	In the classical additive setting, a Lipschitz perturbation on $Z$
	can be extended to $X$ by the McShane extension theorem
	\cite{McShane1934}. In the present setting, a perturbation is represented
	by an orbit Lipschitz almost additive sequence, and extending its terms
	separately may not preserve almost additivity.
	
	For this reason, we work with the extendable space
	\[
	\cE_{\rm orb}^{X}(Z,T)
	:=
	R_Z\bigl(\cE_{\rm orb}(X,T)\bigr)
	\subset
	\cE_{\rm orb}(Z,T).
	\]
	Its elements are precisely the classes on $Z$ that admit a lift to $X$,
	and the lifting norm provides a lift with controlled norm. Therefore,
	the subsystem reduction is first established on
	$\cE_{\rm orb}^{X}(Z,T)$.
	
	The equality
	\[
	\cE_{\rm orb}^{X}(Z,T)=\cE_{\rm orb}(Z,T)
	\]
	holds if and only if $R_Z$ is surjective. In this case, the reduction
	has the same form as in the classical setting. In particular, an
	equivariant Lipschitz retraction from $X$ onto $Z$ guarantees the
	surjectivity of $R_Z$.
\end{remark}
The preceding subsystem reduction is stated on the extendable space $\cE_{\rm orb}^{X}(Z,T)$. We now give a condition under which this space equals the full subsystem space $\cE_{\rm orb}(Z,T)$ and the two norms equivalent.
\begin{theorem}
	\label{x:thm:surjective-restriction}
	If
	$
	R_Z:\cE_{\rm orb}(X,T)\to\cE_{\rm orb}(Z,T)
	$ 
	is surjective, then $ \cE_{\rm orb}^{X}(Z,T)=\cE_{\rm orb}(Z,T),$ and their norms are equivalent.
\end{theorem}

\begin{proof}
	Since $R_Z$ is surjective, $
	\cE_{\rm orb}^{X}(Z,T)=R_Z(\cE_{\rm orb}(X,T))=\cE_{\rm orb}(Z,T).$ 
	The restriction map is a bounded surjection between Banach spaces. 
	By the open mapping theorem, there exists $C_Z<\infty$ such that 
	every $\zeta\in\cE_{\rm orb}(Z,T)$ has a lift $\theta\in\cE_{\rm orb}(X,T)$ satisfying	$R_Z\theta=\zeta
	$ and $\|\theta\|_{\cE_{\rm orb}(X,T)}\leq  C_Z\|\zeta\|_{\cE_{\rm orb}(Z,T)}.$ 
	Therefore, $\|\zeta\|_{Z\leftarrow X}\leq
	\|\theta\|_{\cE_{\rm orb}(X,T)}\leq 	C_Z\|\zeta\|_{\cE_{\rm orb}(Z,T)}.$ 
	Together with \eqref{x:eq:two-restriction-norms}, this gives
	$\|\zeta\|_{\cE_{\rm orb}(Z,T)}\leq\|\zeta\|_{Z\leftarrow X}\leq
	C_Z\|\zeta\|_{\cE_{\rm orb}(Z,T)}.$ 
	Hence the two norms are equivalent.
\end{proof}
We now give a condition under which the restriction map is surjective and has a bounded right inverse.
\begin{definition}
	A map $\pi_Z:X\to Z$ is an \emph{equivariant Lipschitz retraction} if
	$$
	\pi_Z|_Z=\operatorname{id}_Z,
	\qquad \pi_Z\circ T=T|_Z\circ\pi_Z,
	\qquad \Lip(\pi_Z)<\infty.
	$$
\end{definition}

\begin{theorem}\label{x:thm:retraction-lifting}
	Suppose $Z$ admits an equivariant Lipschitz retraction $\pi_Z$.  Then
	$R_Z$ is surjective and has the bounded linear right inverse
	\begin{equation*}
		J_Z:\cE_{\rm orb}(Z,T)\longrightarrow\cE_{\rm orb}(X,T),
		\qquad
		J_Z[(\phi_n)]=[(\phi_n\circ\pi_Z)].
	\end{equation*}
	Moreover,
	\begin{equation*}
		\|J_Z\zeta\|_{\cE_{\rm orb}(X,T)}
		\leq\max\{1,\Lip(\pi_Z)\}
		\|\zeta\|_{\cE_{\rm orb}(Z,T)}.
	\end{equation*}
\end{theorem}

\begin{proof}
	Let $\Phi=\{\phi_n\}_{n\geq1}
	\in\mathcal{AA}_{\rm orb}(Z,T)$ 
	and define $\widetilde\Phi:=
	\{\widetilde\phi_n\}_{n\geq1},$ $	\widetilde\phi_n:=\phi_n\circ\pi_Z.$ 
	Since $\pi_Z|_Z=\operatorname{id}_Z$, the map $\pi_Z:X\to Z$ is
	surjective. We first show that
	$\widetilde\Phi\in\mathcal{AA}_{\rm orb}(X,T)$.
	
	For every $f\in C(Z)$, 
	$\pi_Z\circ T=T\circ\pi_Z$ gives $$S_n(f\circ\pi_Z)(x) =\sum_{j=0}^{n-1}f(T^j(\pi_Zx))=(S_nf)(\pi_Zx).$$ 	Let $\varepsilon>0$. Since $\Phi$ is asymptotically additive on $Z$,
	there exists $f_\varepsilon\in C(Z)$ such that
	$
	\limsup_{n\to\infty}
	\frac1n
	\|\phi_n-S_nf_\varepsilon\|_{\infty,Z}
	\leq\varepsilon.
	$ Then,  
	$
	\|\widetilde\phi_n-S_n(f_\varepsilon\circ\pi_Z)\|_{\infty,X} 
	=\|\phi_n-S_nf_\varepsilon\|_{\infty,Z}.
	$
	Therefore,
	$
	\limsup_{n\to\infty}
	\frac1n
	\|\widetilde\phi_n-S_n(f_\varepsilon\circ\pi_Z)\|_{\infty,X}
	\leq\varepsilon.
	$
	Thus $\widetilde\Phi$ is asymptotically additive on $X$.
	
	We next estimate its almost additivity bound. For every $m,n\geq1$
	and $x\in X$,  
	$$
	\widetilde\phi_{n+m}(x)
	-\widetilde\phi_n(x)
	-\widetilde\phi_m(T^nx)=
	\phi_{n+m}(\pi_Zx)
	-\phi_n(\pi_Zx)
	-\phi_m(T^n(\pi_Zx)).$$ 
	Taking the supremum over $x\in X$ and   $m,n\geq1$, 
	$
	D(\widetilde\Phi)\leq D(\Phi).$ 
	Moreover, since $\pi_Z(X)=Z$,
	$$
	\|\widetilde\phi_n\|_{\infty,X}
	=
	\|\phi_n\|_{\infty,Z}.
	$$
	Consequently,
	$
	\|\widetilde\Phi\|_{\mathcal A(X,T)}
	\leq
	\|\Phi\|_{\mathcal A(Z,T)}.
	$
	It remains to prove the orbit Lipschitz condition. Let $x,y\in X$. 
	\begin{align*}
		|\widetilde\phi_n(x)-\widetilde\phi_n(y)|
		\leq
		L_{\rm orb}(\Phi)
		\sum_{j=0}^{n-1}
		d(T^j(\pi_Zx),T^j(\pi_Zy))
		\leq
		\Lip(\pi_Z)L_{\rm orb}(\Phi)
		\sum_{j=0}^{n-1}
		d(T^jx,T^jy).
	\end{align*}
	Thus
	$
	L_{\rm orb}(\widetilde\Phi)
	\leq
	\Lip(\pi_Z)L_{\rm orb}(\Phi).
	$	It follows that $
	\|\widetilde\Phi\|_{\mathcal{AA}_{\rm orb}(X,T)}
	\leq
	\max\{1,\Lip(\pi_Z)\}
	\|\Phi\|_{\mathcal{AA}_{\rm orb}(Z,T)}.$
	
	We next show that the map is well defined on the quotient spaces.
	Denote by $\mathcal N_Z$ and $\mathcal N_X$ the uniformly sublinear
	spaces on $Z$ and $X$, respectively. If
	$
	\Gamma=\{\gamma_n\}_{n\geq1}
	\in
	\mathcal{AA}_{\rm orb}(Z,T)\cap\mathcal N_Z,
	$
	then $$\limsup_{n\to\infty}
	\frac1n\|\gamma_n\circ\pi_Z\|_{\infty,X}
	=\limsup_{n\to\infty}
	\frac1n\|\gamma_n\|_{\infty,Z}=0.$$
	Hence
	$
	\{\gamma_n\circ\pi_Z\}_{n\geq1}
	\in
	\mathcal{AA}_{\rm orb}(X,T)\cap\mathcal N_X.
	$
	Therefore, if $[\Phi]=[\Psi]$ in $\cE_{\rm orb}(Z,T)$, then
	$
	[(\phi_n\circ\pi_Z)]
	=
	[(\psi_n\circ\pi_Z)]
	$
	in $\cE_{\rm orb}(X,T)$. Thus $J_Z$ is well defined and linear.
	
	Let $\zeta=[\Phi]\in\cE_{\rm orb}(Z,T)$. For every
	$\Gamma\in\mathcal{AA}_{\rm orb}(Z,T)\cap\mathcal N_Z$,
	we have 
	\begin{align*}
		\|J_Z\zeta\|_{\cE_{\rm orb}(X,T)}
		\leq
		\|(\Phi+\Gamma)\circ\pi_Z\|
		_{\mathcal{AA}_{\rm orb}(X,T)}
		\leq
		\max\{1,\Lip(\pi_Z)\}
		\|\Phi+\Gamma\|_{\mathcal{AA}_{\rm orb}(Z,T)}.
	\end{align*}
	Taking the infimum over all such $\Gamma$, we have
	$
	\|J_Z\zeta\|_{\cE_{\rm orb}(X,T)}
	\leq
	\max\{1,\Lip(\pi_Z)\}
	\|\zeta\|_{\cE_{\rm orb}(Z,T)}.
	$
	
	Finally, for every $z\in Z$ and $n\geq1$,
	$$R_ZJ_Z[\Phi]=R_Z[(\phi_n\circ\pi_Z)]=
	[((\phi_n\circ\pi_Z)|_Z)]=[(\phi_n)].$$
	Hence
	$
	R_ZJ_Z=\operatorname{id}_{\cE_{\rm orb}(Z,T)}.
	$
	Thus $J_Z$ is a bounded linear right inverse of $R_Z$. In particular,
	$R_Z$ is surjective.
\end{proof}

We next consider almost additive potentials near a periodic orbit. We estimate
the bound between the averages of an invariant measure and the periodic
measure by the average distance from the orbit. This estimate will be used to
prove global locking.

Let $\gPer(X,T)$ denote the set of periodic invariant measures.  We say 
$\eta\in\cE_{\rm orb}$ has \emph{periodic optimization} if
$
\Max(\eta)=\{\mu_Q\}
$ 
for some periodic orbit $Q$.  It has \emph{weak periodic optimization} if
$\Max(\eta)\cap\gPer(X,T)\neq\varnothing$. 
Let $ Q=\{q_0,q_1,\ldots,q_{p-1}\}$, $Tq_i=q_{i+1\pmod p},$
be a periodic orbit, and let
$
\mu_Q:=\frac1p\sum_{i=0}^{p-1}\delta_{q_i}.
$
Put
$
d_Q(x):=\dist(x,Q).
$
The additive sequence $S_\bullet d_Q$ belongs to
$\mathcal{AA}_{\rm orb}(X,T)$, with bound zero and orbit Lipschitz constant at
most one.

\begin{lemma}\label{g:lem:scalar-deviation}
	Suppose $\{b_n\}_{n\geq1}\subset\R$ and $D\geq0$ satisfy for $n,m\geq1$, 
	$ |b_{n+m}-b_n-b_m|\leq D. $
	If $a=\lim_{n\to\infty}b_n/n$, then for $n\geq1$,  
	$|b_n-na|\leq D$. 
\end{lemma}

\begin{proof}
	For fixed $n$, applying the above inequality repeatedly,  
	$
	|b_{kn}-kb_n|\leq(k-1)D.
	$
	Divide by $k$ and let $k\to\infty$.  Since
	$b_{kn}/k=n\,b_{kn}/(kn)\to na$.
\end{proof}
The next lemma gives a uniform bound for  $\Phi$ on a
periodic orbit.
\begin{lemma}\label{g:lem:phase}
	Let $(X,T)$ be a TDS, and  
	$
	Q=\{q_0,q_1,\ldots,q_{p-1}\}
	$
	be a periodic orbit of period $p$. Let $\Phi=\{\phi_n\}\in\mathcal{AA}_{\rm orb}(X,T)$ and put
	$
	A:=\|\Phi\|_{\cA},$ $D:=D(\Phi),$ $a_Q:=\Phi_*(\mu_Q).$
	Then, for all $n\geq1$ and $0\leq i<p$,
	\begin{equation*}
		|\phi_n(q_i)-na_Q|
		\leq K_Q(\Phi):=(2p+1)(A+D).
	\end{equation*}
\end{lemma}

\begin{proof}
	Set
	$
	b_n:=\int\phi_n\,d\mu_Q
	=\frac1p\sum_{i=0}^{p-1}\phi_n(q_i).
	$
	Since $\mu_Q$ is invariant, almost additivity gives
	$
	|b_{n+m}-b_n-b_m|\leq D.
	$
	Lemma~\ref{g:lem:scalar-deviation} therefore yields
	\begin{equation}\label{g:eq:b-linear}
		|b_n-na_Q|\leq D.
	\end{equation}
	
	Applying almost additivity with the pairs $(1,n)$ and $(n,1)$, we have 
	$|\phi_{n+1}(q_i)-\phi_1(q_i)-\phi_n(q_{i+1})|\leq D$, $|\phi_{n+1}(q_i)-\phi_n(q_i)-\phi_1(T^nq_i)|\leq D.$
	By the triangle inequality,
	$
	|\phi_n(q_{i+1})-\phi_n(q_i)|
	\leq2\|\phi_1\|_\infty+2D
	\leq2(A+D).
	$ Thus for $i,j\in\mathbb{N}$, $|\phi_n(q_{i})-\phi_n(q_j)|
	\leq2p(A+D)$. 
	Since $Q$ contains $p$ points, for every $0\leq i,j<p$,
	$
	|\phi_n(q_i)-b_n|\leq2p(A+D).
	$
	Together with \eqref{g:eq:b-linear}, yields the proof.
\end{proof}

Choose $\rho_Q>0$ such that the balls
$B(q_i,\rho_Q)$ are pairwise disjoint and
\begin{equation}\label{g:eq:phase-coherence}
	x\in B(q_i,\rho_Q),\quad Tx\in B(Q,\rho_Q)
	\quad\Longrightarrow\quad
	Tx\in B(q_{i+1},\rho_Q).
\end{equation}
Suppose first that $p\geq2$ and set
$
s:=\min_{i\neq j}d(q_i,q_j)>0.
$
By continuity of $T$, we may choose $0<\rho_Q<s/3$ such that
$
T(B(q_i,\rho_Q))\subset B(q_{i+1},s/3)
$
for every $i$. If $Tx\in B(q_j,\rho_Q)$ for some $j\neq i+1$, then
$$
s\leq d(q_j,q_{i+1})
<\rho_Q+s/3<2s/3.
$$
 Hence \eqref{g:eq:phase-coherence} holds. If $p=1$, the result is clear.

For $x\in X$ and $n\geq1$, call $j\in\{0,\ldots,n-1\}$ \emph{bad} if $ \dist(T^jx,Q)\geq\rho_Q, $ and \emph{good} otherwise. Write $b_n(x)$ for the number of bad times. Then 
\begin{equation*}
	 b_n(x) \leq\frac1{\rho_Q} \sum_{j=0}^{n-1}\dist(T^jx,Q).
\end{equation*} 
Divide $\{0,\ldots,n-1\}$ into maximal consecutive good and bad blocks. There are at most $b_n(x)+1$ good blocks, and the total number $r_n(x)$ of blocks satisfies \begin{equation}\label{g:eq:block-count} r_n(x)\leq2b_n(x)+1. 
\end{equation}
The following theorem is proved directly for an almost additive sequence
$\Phi=\{\phi_n\}_{n\geq1}$, without reducing it to an additive sequence
$S_\bullet f$. It combines the estimates above to give a uniform bound for
the difference between the asymptotic averages of $\nu$ and $\mu_Q$. This
bound will be used to prove global locking.
\begin{theorem}\label{thm:intro-global-locking}

	Let $(X,T)$ be a TDS, for every periodic orbit $Q$ there exists $C_Q>0$ such that, for every
	$\Phi\in\mathcal{AA}_{\rm orb}(X,T)$ and every
	$\nu\in\M(X,T)$,
	\begin{equation}\label{g:eq:representative-domination}
		|\Phi_*(\nu)-\Phi_*(\mu_Q)|
		\leq C_Q\|\Phi\|_{\mathcal{AA}_{\rm orb}}
		\int\dist(x,Q)\,d\nu(x).
	\end{equation}
	Consequently, for every $\eta\in\cE_{\rm orb}$,
	\begin{equation*}
		|\eta_*(\nu)-\eta_*(\mu_Q)|
		\leq C_Q\|\eta\|_{\cE_{\rm orb}}
		\int\dist(x,Q)\,d\nu(x).
	\end{equation*}
\end{theorem}

\begin{proof}
	Fix $\Phi=\{\phi_n\}\in\mathcal{AA}_{\rm orb}(X,T)$ and abbreviate
	$A=\|\Phi\|_{\cA}$, $D=D(\Phi)$, $ L=L_{\rm orb}(\Phi)$,  $N_\Phi=A+D+L$. 
	Also write $a_Q=\Phi_*(\mu_Q)$ and
	$K_Q=K_Q(\Phi)$ as in Lemma~\ref{g:lem:phase}.
	
	Let $\{r,\ldots,r+\ell-1\}$ be a maximal consecutive interval such that
	$
	T^{r+j}x\in B(Q,\rho_Q)
	$
	for every $0\leq j<\ell$. 
	By \eqref{g:eq:phase-coherence}, there exists  $i$ such that  
	$
	\dist(T^{r+j}x,Q)=d(T^{r+j}x,q_{i+j\pmod p})
	$ for all $j\in\{0,\cdots,\ell-1\}$. 
	By the definition of $L=L_{\rm orb}(\Phi)$  and Lemma~\ref{g:lem:phase},
	\begin{equation}\label{g:eq:good-block}
		|\phi_\ell(T^rx)-\ell a_Q|
		\leq K_Q+L\sum_{j=0}^{\ell-1}
		\dist(T^{r+j}x,Q).
	\end{equation}	
	For a bad block of length $\ell$, the definition of $A=\|\Phi\|_{\cA}$ gives
	$
	\|\phi_\ell\|_\infty\leq\ell A,
	$ and
	$|a_Q|\leq A$ give
	\begin{equation}\label{g:eq:bad-block}
		|\phi_\ell(T^rx)-\ell a_Q|
		\leq2A\ell.
	\end{equation}	
	Recall that $r_n(x)$ denotes the total number of maximal good and bad
	blocks. By almost additivity,
	\begin{equation*}
		|\phi_n(x)-\sum_{s=1}^{r_n(x)}
		\phi_{\ell_s}(T^{t_s}x)|\leq
		(r_n(x)-1)D.
	\end{equation*}  Combining \eqref{g:eq:block-count},
	\eqref{g:eq:good-block}   and \eqref{g:eq:bad-block}, we have
	\begin{align*}
		|\phi_n(x)-na_Q|
		&\leq (b_n(x)+1)K_Q
		+L\sum_{j=0}^{n-1}\dist(T^jx,Q)
		+2A b_n(x)+2D b_n(x)\notag\\
		&\leq K_Q+
		\left(L+\frac{K_Q+2A+2D}{\rho_Q}\right)
		\sum_{j=0}^{n-1}\dist(T^jx,Q).
	\end{align*}
	By Lemma~\ref{g:lem:phase},
	$K_Q\leq(2p+1)N_\Phi$,  
	$K_Q+2A+2D\leq(2p+3)N_\Phi.$ 
	Thus  
	\begin{equation*}
		|\phi_n(x)-na_Q|
		\leq(2p+1)N_\Phi
		+\left(1+\frac{2p+3}{\rho_Q}\right)N_\Phi
		\sum_{j=0}^{n-1}\dist(T^jx,Q).
	\end{equation*}	
	Integrate with respect to $\nu$, 
	$$
	\int\sum_{j=0}^{n-1}\dist(T^jx,Q)\,d\nu(x)
	=n\int\dist(x,Q)\,d\nu(x).
	$$
	Divide by $n$ and let $n\to\infty$.  Thus, \eqref{g:eq:representative-domination} holds with
	$
	C_Q=1+(2p+3)/\rho_Q.
	$
	Finally, let $\eta\in\cE_{\rm orb}$.  For every $\delta>0$, choose  $\Phi$ such that
	$
	[\Phi]=\eta$ and $
	\|\Phi\|_{\mathcal{AA}_{\rm orb}}
	<\|\eta\|_{\cE_{\rm orb}}+\delta.
	$
	If $[\Psi]=\eta$, then
	$\Gamma=\Psi-\Phi=\{\gamma_n\}_{n\geq1}\in\mathcal N$ 
	Hence, for every $\mu\in\M(X,T)$, 		$$|\Psi_*(\mu)-\Phi_*(\mu)|
	\leq\lim_{n\to\infty}\frac1n\|\gamma_n\|_\infty=0.$$ 
	Thus $\eta_*(\mu)=\Phi_*(\mu)$.
	Apply \eqref{g:eq:representative-domination} and let
	$\delta\to0$ to give the proof.
\end{proof}

We next prove the global orbit Lipschitz locking theorem. It shows that if
$\mu_Q$ maximizes $\eta$, then subtracting a small multiple of the distance
from $Q$ makes $\mu_Q$ uniquely maximizing. Moreover, this property remains
valid under all sufficiently small perturbations in $\cE_{\rm orb}$
\begin{theorem}\label{g:thm:global-locking}
	Let $(X,T)$ be TDS, and $\eta\in\cE_{\rm orb}(X)$ and suppose
	$\mu_Q\in\Max(\eta)$.  For $\varepsilon>0$, define 
	$
		\eta_\varepsilon
		:=\eta-\varepsilon[S_\bullet d_Q].
	$
	Then $\mu_Q$ is the unique maximizing measure for every
	$\eta_\varepsilon+\theta$ satisfying
	\begin{equation}\label{g:eq:locking-ball}
		\|\theta\|_{\cE_{\rm orb}}<\frac{\varepsilon}{C_Q},
	\end{equation}
	where $C_Q$ is the constant in
	Theorem~\ref{thm:intro-global-locking}.  In particular, the  periodic
	measure $\mu_Q$ remains uniquely maximizing throughout an open neighborhood
	of $\eta_\varepsilon$.
\end{theorem}

\begin{proof}
	Let $\nu\in\M(X,T)$.  Maximality of $\mu_Q$ for $\eta$ and
	Theorem~\ref{thm:intro-global-locking} give
	\begin{align*}
		(\eta_\varepsilon+\theta)_*(\nu)
		-(\eta_\varepsilon+\theta)_*(\mu_Q)
		&\leq-\varepsilon\int d_Q\,d\nu
		+C_Q\|\theta\|_{\cE_{\rm orb}}\int d_Q\,d\nu\\
		&=-\bigl(\varepsilon-C_Q\|\theta\|_{\cE_{\rm orb}}\bigr)  \int d_Q\,d\nu.
	\end{align*}
	Under \eqref{g:eq:locking-ball}, this is strictly negative unless
	$\nu$ is supported on $Q$. Since $mu_Q$ is the only invariant probability measure supported
	on the  periodic orbit $Q$.  Hence $\mu_Q$ is uniquely maximizing.
\end{proof}

We now give several equivalent criteria for global TPO in
$\cE_{\rm orb}$. We also introduce $X$-extendable TPO for a closed
invariant subsystem $Z$. The locking theorem and the global subsystem
reduction theorem are then used to transfer periodic optimization from
$Z$ to an open set $U\subset\cE_{\rm orb}(X,T)$ such that $Z$ is
$\eta$-maximizable for every $\eta\in U$.

\begin{definition}
	Let $(X,T)$ be a TDS and let $\eta\in\cE_{\rm orb}$.
	We say that $\eta$ has the \emph{periodic optimization property} if there
	exists   $\mu_Q\in \gPer(X,T)$ such that
	$$
	\Max(\eta)=\{\mu_Q\}.
	$$
	We say that $\eta$ has the \emph{weak periodic optimization property} if
	there exists  $\mu_Q\in \gPer(X,T)$ such that
	$$
	\mu_Q\in\Max(\eta),
	$$
	or equivalently,
	$
	\Max(\eta)\cap\Per(X,T)\neq\varnothing.
	$
	We say that $(X,T)$ has \emph{global TPO} in $\cE_{\rm orb}$ if there
	exists an open dense set $$\mathcal G\subset\cE_{\rm orb}$$ such that every
	$\eta\in\mathcal G$ has the periodic optimization property. Here openness
	and density are taken with respect to the norm of $\cE_{\rm orb}$.
\end{definition}
To compare periodic optimization, weak periodic optimization, and locking,
we introduce the following subsets of $\cE_{\rm orb}$.
\begin{align*}
	\mathcal U_!
	&:=\{\eta\in\cE_{\rm orb}:\eta\text{ has periodic optimization}\},\\
	\mathcal U_{\rm w}
	&:=\{\eta\in\cE_{\rm orb}:\eta\text{ has weak periodic optimization}\},\\
	\mathcal P
	&:=\interior(\mathcal U_!),\quad
	\mathcal P_+
	:=\interior(\mathcal U_{\rm w}),\quad
	\mathcal P_-
	:=\bigcup_{\mu\in\gPer(X,T)}
	\interior\{\eta\in\cE_{\rm orb}:\Max(\eta)=\{\mu\}\}.
\end{align*}
Clearly, 
\begin{equation}\label{g:eq:P-inclusions-global}
	\mathcal P_-
	\subset\mathcal P
	\subset\mathcal P_+
	\subset\mathcal U_{\rm w}.
\end{equation}
To understand how strong the openness condition in global TPO can be, we
compare the classes above. The strongest class, $\mathcal P_-$, requires one
fixed periodic measure to remain uniquely maximizing throughout an open
neighborhood. The following corollary shows that all these classes have the
same closure.
\begin{corollary}\label{g:cor:closure-equality}
	Let $(X,T)$ be a TDS. Then the sets defined above have the same closure in
	$\cE_{\rm orb}$, 
	$$
	\cl(\mathcal P_-)
	=\cl(\mathcal P)
	=\cl(\mathcal P_+)
	=\cl(\mathcal U_{\rm w}).
	$$
\end{corollary}
\begin{proof}
	It is enough to prove
	$\mathcal U_{\rm w}\subset\cl(\mathcal P_-)$. 	
	Let $\eta\in\mathcal U_{\rm w}$. By the definition of weak periodic
	optimization, there exists a periodic orbit $Q$ such that
	$
	\mu_Q\in\Max(\eta).
	$
	For every $\varepsilon>0$, set
	$
	\eta_\varepsilon=\eta-\varepsilon[S_\bullet d_Q].
	$
	By Theorem~\ref{g:thm:global-locking}, $\mu_Q$ is the unique maximizing
	measure for every element in an open neighborhood of
	$\eta_\varepsilon$. Therefore,
	$$
	\eta_\varepsilon
	\in
	\interior
	\{\zeta\in\cE_{\rm orb}:\Max(\zeta)=\{\mu_Q\}\}
	\subset\mathcal P_-.
	$$
	Moreover, $$\|\eta_\varepsilon-\eta\|_{\cE_{\rm orb}}
	=\varepsilon\|[S_\bullet d_Q]\|_{\cE_{\rm orb}}\to0\quad \text{ as } \varepsilon\to0.$$
	Hence
	$\eta_\varepsilon\longrightarrow\eta\in\cE_{\rm orb}$, 
	and so $\eta\in\cl(\mathcal P_-)$. Thus,
	$
	\mathcal U_{\rm w}\subset\cl(\mathcal P_-).
	$
	Combining this inclusion with
	\eqref{g:eq:P-inclusions-global} proves the result.
\end{proof}

The equality of closures gives several equivalent ways to verify global TPO.
In particular, it is enough to prove that weak periodic optimization is
dense; global locking then gives density of the strongest open class
$\mathcal P_-$. This yields the following equivalenc
\begin{corollary}\label{g:cor:TPO-equivalence}
	The following are equivalent:
	\begin{enumerate}[label=\textup{(\roman*)}]
		\item $(X,T)$ has global TPO in $\cE_{\rm orb}$;
		\item $\mathcal P$ is dense in $\cE_{\rm orb}$;
		\item $\mathcal U_!$ is dense in $\cE_{\rm orb}$;
		\item $\mathcal P_+$ is dense in $\cE_{\rm orb}$;
		\item $\mathcal U_{\rm w}$ is dense in $\cE_{\rm orb}$;
		\item $\mathcal P_-$ is dense in $\cE_{\rm orb}$.
	\end{enumerate}
\end{corollary}

\begin{proof}
	Use Corollary~\ref{g:cor:closure-equality} and the fact that
	$\mathcal P=\interior(\mathcal U_!)$ is the largest open subset of
	$\mathcal U_!$.
\end{proof}
The subsystem reduction theorem is stated on the extendable space
$\cE_{\rm orb}^{X}(Z,T)$. To apply this theorem to periodic optimization,
we introduce $X$-extendable TPO on this space. 
\begin{definition}
	Let $(X,T)$ be a TDS, and let $Z\subset X$ be a non-empty closed
	invariant set.	
	A class $\zeta\in\cE_{\rm orb}^{X}(Z,T)$ has the
	\emph{weak periodic optimization property} if
	$\mathcal M_{\max,Z}(\zeta)\cap\gPer(Z,T)\neq\varnothing.$ 
	Define
	$$
	\mathcal U_{\rm w}^{X}(Z)
	:=
	\left\{
	\zeta\in\cE_{\rm orb}^{X}(Z,T):
	\mathcal M_{\max,Z}(\zeta)\cap\gPer(Z,T)\neq\varnothing
	\right\}.
	$$
	A class $\zeta\in\cE_{\rm orb}^{X}(Z,T)$ has the
	\emph{periodic optimization property} if
	$$
	\mathcal M_{\max,Z}(\zeta)=\{\mu_Q\}
	$$
	for some $\mu_Q\in\gPer(Z,T)$.
	
	We say that $(Z,T)$ has \emph{$X$-extendable TPO} if there exists an
	open and dense set
	$$
	\mathcal G_Z^{X}\subset\cE_{\rm orb}^{X}(Z,T)
	$$
	such that every $\zeta\in\mathcal G_Z^{X}$ has the periodic
	optimization property. Here openness and density are taken with respect
	to the lifting norm $\|\cdot\|_{Z\leftarrow X}$.
\end{definition}
If $Q\subset Z$ is a periodic orbit, then
$
[S_\bullet\dist_X(\cdot,Q)]|_Z
\in\cE_{\rm orb}^{X}(Z,T),
$
because it is the restriction of
$[S_\bullet\dist_X(\cdot,Q)]\in\cE_{\rm orb}(X,T)$.
The periodic domination theorem on $Z$, together with
\eqref{x:eq:two-restriction-norms}, gives the same locking estimate in the
lifting norm. The next lemma therefore gives equivalent conditions for
$X$-extendable TPO.

\begin{lemma}
	\label{x:lem:extendable-TPO-equivalence}
	The following are equivalent:
	\begin{enumerate}[label=\textup{(\roman*)}]
		\item $(Z,T)$ has $X$-extendable TPO;
		\item $\mathcal U_!^{X}(Z)$ is dense in
		$\cE_{\rm orb}^{X}(Z,T)$;
		\item $\mathcal U_{\rm w}^{X}(Z)$ is dense in
		$\cE_{\rm orb}^{X}(Z,T)$.
	\end{enumerate}
\end{lemma}

\begin{proof}
	The implications \textup{(i)}$\Rightarrow$\textup{(ii)}
	and \textup{(ii)}$\Rightarrow$\textup{(iii)} follow directly from
	the definitions.
	
	Let
	$\zeta\in\mathcal U_{\rm w}^{X}(Z)$ and choose a periodic orbit
	$Q\subset Z$ such that
	$\mu_Q\in\mathcal M_{\max,Z}(\zeta)$. For $t>0$, set
	$$
	\zeta_t
	:=
	\zeta-t[S_\bullet\dist_X(\cdot,Q)]|_Z.
	$$
	The global locking theorem applied to $(Z,T)$, together with
	$
	\|\theta\|_{\cE_{\rm orb}(Z,T)}
	\leq
	\|\theta\|_{Z\leftarrow X},
	$
	shows that $\zeta_t$ belongs to
	$\interior\bigl(\mathcal U_!^{X}(Z)\bigr)$.
	
	Moreover,
	$$
	\|\zeta_t-\zeta\|_{Z\leftarrow X}
	\longrightarrow0
	\qquad\text{as }t\to0.
	$$
	Hence
	$
	\mathcal U_{\rm w}^{X}(Z)
	\subset
	\cl\bigl(\interior(\mathcal U_!^{X}(Z))\bigr).
	$
	Since $\mathcal U_{\rm w}^{X}(Z)$ is dense,
	$\interior(\mathcal U_!^{X}(Z))$ is open and dense in
	$\cE_{\rm orb}^{X}(Z,T)$. Therefore, $(Z,T)$ has
	$X$-extendable TPO.
\end{proof}

\begin{remark}
	Global TPO and $X$-extendable TPO have the same form, but they are defined
	on different spaces and have different roles. 
	There are three cases.
	\begin{enumerate}[label=\textup{(\roman*)}]
		\item If $Z=X$, then
		$
		\cE_{\rm orb}^{X}(X,T)=\cE_{\rm orb}(X,T),
		$
		and the lifting norm is the norm of $\cE_{\rm orb}(X,T)$. Hence
		$X$-extendable TPO is exactly global TPO.
		
		\item Suppose that $Z\subsetneq X$ and $R_Z$ is surjective. Then
		$$
		\cE_{\rm orb}^{X}(Z,T)=\cE_{\rm orb}(Z,T),
		$$
		and their norms are equivalent by
		Theorem~\ref{x:thm:surjective-restriction}. Therefore,
		$X$-extendable TPO is equivalent to TPO of $(Z,T)$ in
		$\cE_{\rm orb}(Z,T)$. However, this is still a property of the
		subsystem $(Z,T)$ and does not by itself give global TPO on $X$.
		
		\item If $R_Z$ is not surjective, then
		$$
		\cE_{\rm orb}^{X}(Z,T)
		=
		R_Z(\cE_{\rm orb}(X,T))
		\subsetneq
		\cE_{\rm orb}(Z,T).
		$$
		In this case, $X$-extendable TPO is defined only for the classes on
		$Z$ that have a lift to $X$. It does not imply TPO in  
		$\cE_{\rm orb}(Z,T)$ without an additional assumption.
	\end{enumerate}
\end{remark}

Combining the global subsystem reduction theorem with global locking, we
obtain the following transfer of periodic optimization from $Z$ to $X$.
\begin{corollary}\label{x:cor:local-periodic-transfer}
	Under the hypotheses on $U$ and $Z$ in
	Theorem~\ref{x:thm:subsystem-reduction-extendable}, suppose that $(Z,T)$ has
	$X$-extendable TPO.  Then
	\begin{equation*} 
		\mathcal P\cap U\quad\text{is dense in }U.
	\end{equation*}
\end{corollary}

\begin{proof}
	Apply Theorem~\ref{x:thm:subsystem-reduction-extendable} with
	$\mathscr N=\gPer(X,T)$.  Since
	$\gPer(X,T)\cap\M(Z,T)=\gPer(Z,T)$, weak periodic optimization is dense in
	$U$.  The equality of closures in Corollary~\ref{g:cor:closure-equality},
	applied inside sufficiently small open balls contained in $U$, gives the proof.
\end{proof}

If $R_Z$ is surjective, Theorem~\ref{x:thm:surjective-restriction} shows that
the  norm on $\cE_{\rm orb}(Z,T)$ and the lifting norm are equivalent.
Hence  global TPO of $(Z,T)$ implies $X$-extendable TPO, and
Corollary~\ref{x:cor:local-periodic-transfer} applies.  By
Theorem~\ref{x:thm:retraction-lifting}, the same conclusion holds when $Z$
admits an equivariant Lipschitz retraction from $X$.

The implications above can be summarized as follows:
\begin{equation*}
	\substack{
		(Z,T)\text{ has TPO in }\cE_{\rm orb}(Z,T)\\
		R_Z\text{ is surjective}
	}
	\Longrightarrow
	\substack{
		(Z,T)\text{ has}\\
		X\text{-extendable TPO}
	}
	\Longrightarrow
	\substack{
		\mathcal U_{\rm w}\cap U\\
		\text{is dense in }U
	}
	\Longrightarrow
	\substack{
		\mathcal P\cap U\\
		\text{is dense in }U.
	}
\end{equation*}

Combining the Baire reduction with the local periodic transfer result, we
give the proof of Theorem \ref{thm:intro-global-structural}.

\begin{proof}[Proof of Theorem~\ref{thm:intro-global-structural}]
	Put $U_i:=\interior(\cE_{Z_i})$.  By
	Theorem~\ref{g:thm:Baire-global}, $\bigcup_{i\geq0}U_i$ is dense in $\cE_{\rm orb}(X,T)$.
	For every $i\geq1$, the set $Z_i$ is maximizable throughout $U_i$, so
	Corollary~\ref{x:cor:local-periodic-transfer} says that
	$\mathcal P\cap U_i$ is dense in $U_i$.  Therefore
	$U_0\cup\mathcal P$ is dense in $\bigcup_{i\geq0}U_i$, and hence in
	$\cE_{\rm orb}(X,T)$.
\end{proof}

\section{A global TPO example}

In this section, we construct a compact dynamical system from countably many
periodic orbits. We prove that it has global TPO in
$\cE_{\rm orb}(X,T)$. Moreover, $\cE_{\rm orb}(X,T)$ is infinite-dimensional,
and the periods of the maximizing periodic orbits can be arbitrarily large.

For $m\geq1$, set
$ p_m:=m+1,$ $ a_m:=2^{-m}$, $Q_m:=\{(m,j):j\in\mathbb Z/p_m\mathbb Z\}.$ 
Let $Q_0:=\{\ast\}$ and  $Y:=\{\ast\}\sqcup\bigsqcup_{m\geq1}Q_m.$ 
Define 
$
r(\ast)=0$, $
r(m,j)=a_m,
$
and define the  metric
\begin{equation}\label{b:eq:star-metric}
	d_Y(u,v)
	:=
	\begin{cases}
		0,&u=v,\\
		r(u)+r(v),&u\neq v.
	\end{cases}
\end{equation}
Define $R:Y\to Y$ by
$R(\ast)=\ast,$  $R(m,j)=(m,j+1\!\!\pmod {p_m}).$

We begin with the topological and dynamical properties of $Y$ and the dynamics of $R$. These facts
will also be used below to describe the invariant measures of the system.
\begin{lemma}\label{b:lem:Y-compact}
	The space $(Y,d_Y)$ is compact, and $R$ is an isometry.  The orbit $Q_m$ has
	prime period $p_m=m+1$, while $Q_0$ is a fixed point.
\end{lemma}

\begin{proof}
	For $u\in Y$, set
	$
	r(u):=d_Y(u,\ast).
	$
	Then $r(\ast)=0$ and $r(u)=a_m$ for $u\in Q_m$. By
	\eqref{b:eq:star-metric}, for $u\neq v$,
	$
	d_Y(u,v)=r(u)+r(v).
	$
	For any $u,v,w\in Y$,
	$
	d_Y(u,v)\leq d_Y(u,w)+d_Y(w,v).
	$
	Hence $d_Y$ is a metric on $Y$. 
	Let $\{u_n\}_{n\geq1}\subset Y$. If $\{u_n\}_{n\geq1}$ belong to only finitely many sets $Q_m$, then they
	belong to a finite subset of $Y$. Thus $\{u_n\}$ has a constant
	subsequence. Otherwise, there is
	a subsequence $\{u_{n_k}\}$ such that
	$
	u_{n_k}\in Q_{m_k}
	$
	and $m_k\to\infty$. Since $a_m\to0$,
	$
	d_Y(u_{n_k},\ast)=a_{m_k}\to0.
	$
	Hence $u_{n_k}\to\ast$. This proves that $Y$ is compact. For every $u\in Y$,
	$
	r(Ru)=r(u).
	$
	The map $R$ fixes $\ast$ and permutes the points of each $Q_m$. If
	$u\neq v$, then $Ru\neq Rv$, and
	\begin{align*}
		d_Y(Ru,Rv)
		=r(Ru)+r(Rv)
		=r(u)+r(v)
		=d_Y(u,v).
	\end{align*}
	Thus $R$ is an isometry.
	By the definition of $R$,  $Q_m$ has  period $p_m=m+1$ for $m\geq1$, and
	$Q_0=\{\ast\}$ is a fixed point. 
\end{proof}
Fix $0<\lambda<1$ and define, 
\begin{align}\label{b:eq:system}
	X:=Y\times[0,1],\quad
	d_X((u,t),(v,s)):=d_Y(u,v)+|t-s|,\quad 
	T(u,t):=(Ru,\lambda t).	 
\end{align} 
For $m\geq0$,  let $
Z_m:=Q_m\times\{0\}.$ 
When $m=0$ this means $Z_0=\{(\ast,0)\}$.
We first describe the periodic orbits of $(X,T)$. Since the second coordinate
is multiplied by $\lambda<1$, every periodic point must lie in
$Y\times\{0\}$. Thus the periodic orbits are exactly the sets $Z_m$.
\begin{proposition}\label{b:prop:orbits}
	The set $Z_m$ is a periodic orbit of prime period $p_m=m+1$ for $m\geq1$,
	and $Z_0$ is a fixed point.  Every $(u,t)$ with $t>0$ is non-periodic.
\end{proposition}

\begin{proof}
	For every $n\geq0$,
	$
	T^n(u,t)=(R^nu,\lambda^nt).
	$
	If $t>0$ and $n\geq1$, then $\lambda^nt\neq t$, so $(u,t)$ cannot be
	periodic.  The remaining assertions follow from Lemma~\ref{b:lem:Y-compact}.
\end{proof}

The next proposition gives  every invariant measure is
supported on
$
Y\times\{0\}=\bigsqcup_{m\geq0}Z_m,
$
and is a unique convex combination of the periodic measures $\mu_m$. For $m\geq1$, let $ 
\mu_m
:=\frac1{p_m}\sum_{j=0}^{p_m-1}
\delta_{((m,j),0)},$ 
$\mu_0:=\delta_{(\ast,0)}.$

\begin{proposition}
	\label{b:prop:measures}
	Let $(X,T)$ be defined by \eqref{b:eq:system}. For $m\geq0$, let $\mu_m$
	be the periodic measure supported on $Z_m$. In particular,
	$
	\mu_0=\delta_{(\ast,0)}.
	$
	Every $\mu\in\M(X,T)$ has a unique representation
	\begin{equation}\label{b:eq:measure-simplex}
		\mu=\sum_{m=0}^{\infty}w_m\mu_m,
		\qquad
		w_m\geq0,
		\qquad
		\sum_{m=0}^{\infty}w_m=1.
	\end{equation}
	Moreover, $	\E(X,T)=\{\mu_m:m\geq0\},$ $\mu_{m}\to\mu_{0}.$
\end{proposition}

\begin{proof}
	Let $\mu\in\M(X,T)$.  Let $\pi:X\to[0,1]$, $\pi(u,t)=t$, be the  projection map.  Then
	$$
	\int_X\pi\,d\mu
	=\int_X\pi\circ T\,d\mu
	=\lambda\int_X\pi\,d\mu.
	$$
	Hence $\int \pi\,d\mu=0$, and by the non-negativity of $\pi$, $\mu(Y\times \{0\})=1$.    And $Y\times\{0\}$ is the disjoint union of the
	$Q_m\times\{0\}$.  Let 
	$w_m=\mu(Q_m\times\{0\})$, by the $\mu$ is
	supported on $Y\times\{0\}$, this proves
	\eqref{b:eq:measure-simplex}. Since   $Q_m\times\{0\}$ are pairwise disjoint and $\mu_j$ is
	supported on $Q_j\times\{0\}$, this proves uniqueness. Each $\mu_m$ is ergodic since it is the periodic measure on the $Z_m$. Conversely, if $\mu$ is ergodic, then
	$
	w_m=\mu(Z_m)\in\{0,1\}
	$
	for every $m\geq0$. Since
	$
	\sum_{m\geq0}w_m=1,
	$
	there is a unique $m$ such that $w_m=1$. Hence $\mu=\mu_m$.  Finally, each   $(u,v)\in Q_m\times\{0\}$ has $d_{X}((u,v),(\ast,0))=a_m$,  and $a_m\to0$, which finish the proof.
\end{proof}
We next show that the periodic orbits $\{Z_m:m\geq0\}$ form a global
maximizable family. For each $\eta\in\cE_{\rm orb}(X,T)$, an ergodic
$\eta$-maximizing measure is equal to some $\mu_m$ which is supported on
$Z_m$.

For a closed invariant set $Z\subset X$, define
$
\cE_Z
:=\{\eta\in\cE_{\rm orb}(X,T):
\Max(\eta)\cap\M(Z,T)\neq\varnothing\}.
$ 
A collection $\cZ$ of closed invariant sets is called
$\cE_{\rm orb}$-maximizable if
\begin{equation*}
	\cE_{\rm orb}(X,T)=\bigcup_{Z\in\cZ}\cE_Z.
\end{equation*}
Since $\mu_m\to\mu_0$ and the map $\eta_*$ is continuous, this
leads to the following countable maximizable family.
\begin{theorem}\label{b:thm:max-family}
	The family $\cZ:=\{Z_m:m\geq0\}$ 
	is a countable $\cE_{\rm orb}$-maximizable family.
\end{theorem}

\begin{proof}
	Fix $\eta=[\Phi]\in\cE_{\rm orb}(X,T)$, where
	$
	\Phi=\{\phi_k\}_{k\geq1}.
	$. For $m\geq0$, set 
	$$F_m=\eta_*(\mu_m)
	=\lim_{k\to\infty}\frac1k\int_X\phi_k\,d\mu_m.$$   Lemma~\ref{lem:functional} and
	$\mu_m\to\mu_0$ imply $F_m=\eta_*(\mu_m)\longrightarrow F_0=\eta_*(\mu_0).$  
	Hence the set $\{F_m:m\geq0\}$ has a largest element.  Choose $m$ such that
	$F_m=\max_nF_n$.  By the Proposition \ref{b:prop:measures} and $\eta_*$ is continuous and affine,
	\begin{equation*}
		\eta_*\bigg(\sum_{n=0}^{\infty}w_n\mu_n\bigg)
		=\sum_{n=0}^{\infty}w_nF_n
		\leq F_m.
	\end{equation*}
	The interchange  is justified by the uniform
	estimate \eqref{eq:uniform-approx}.  Thus $\mu_m\in\Max(\eta)$, so
	$\eta\in\cE_{Z_m}$.  This complete the proof.
\end{proof}

The next two lemmas apply the general results of global TPO to the sets
$\cE_{Z_m}$. Lemma~\ref{b:lem:EZ-closed} follows from
Lemma~\ref{g:lem:closedness-global}. Lemma~\ref{b:lem:common-interior}
follows from Lemma~\ref{g:lem:common-interior-global} and
$
\M(Z_m,T)=\{\mu_m\}.
$

\begin{lemma}\label{b:lem:EZ-closed}
	For every $m\geq0$, the set $\cE_{Z_m}$ is closed in
	$\cE_{\rm orb}(X,T)$.
\end{lemma}
\begin{proof}
	Since
	$
	\M(Z_m,T)=\{\mu_m\},
	$
	we have
	$
	\cE_{Z_m}
	=
	\bigcap_{\nu\in\M(X,T)}
	\left\{
	\eta\in\cE_{\rm orb}:
	\eta_*(\mu_m)-\eta_*(\nu)\geq0
	\right\}.
	$
	For each $\nu\in\M(X,T)$, define
	$
	F_\nu:\cE_{\rm orb}\longrightarrow\mathbb R,
	$, $
	F_\nu(\eta):=\eta_*(\mu_m)-\eta_*(\nu).
	$ 
	For $\eta,\zeta\in\cE_{\rm orb}$,
	\begin{align*}
		|F_\nu(\eta)-F_\nu(\zeta)|
		&\leq
		|(\eta-\zeta)_*(\mu_m)|
		+
		|(\eta-\zeta)_*(\nu)|\\
		&\leq
		2\|\eta-\zeta\|_{\cE_{\rm orb}}.
	\end{align*}
	Thus $F_\nu$ is continuous and
	$
	\cE_{Z_m}
	=
	\bigcap_{\nu\in\M(X,T)}
	F_\nu^{-1}\bigl([0,\infty)\bigr).
	$
	Therefore, $\cE_{Z_m}$ is closed in $\cE_{\rm orb}(X,T)$.
\end{proof}

Each $Z_m$ supports exactly one invariant measure, namely $\mu_m$. Hence an
interior point of $\cE_{Z_m}$ gives stable unique maximization by $\mu_m$, as
stated in the following lemma.
\begin{lemma}\label{b:lem:common-interior}
	If $\eta\in\interior(\cE_{Z_m})$, then there exists an open neighborhood
	$U\subset\cE_{\rm orb}(X,T)$ of $\eta$ such that
	$	\Max(\zeta)=\{\mu_m\}$ for every $\zeta\in U.
	$
	In particular, 
	\begin{align*}
		\Max(\eta)=\{\mu_m\}.
	\end{align*}
\end{lemma}

\begin{proof}
	We already know that $\mu_m\in\Max(\eta)$.  Suppose that
	$\nu\in\Max(\eta)$ and $\nu\neq\mu_m$.  Write
	$\nu=\sum_nw_n\mu_n$.  Since every $\eta_*(\mu_n)\leq\beta(\eta)$ and the
	weighted average equals $\beta(\eta)$, some $n\neq m$ with $w_n>0$ satisfies
	$\mu_n\in\Max(\eta)$.
	
	Let $D_m(x)=\dist(x,Z_m)$.  This is a Lipschitz function, it vanishes on
	$Z_m$, and it is strictly positive on $Z_n$.  For every $t>0$, the additive
	perturbation $t[S_\bullet D_m]$ raises the energy of $\mu_n$ strictly more
	than that of $\mu_m$.  Consequently,
	$\eta+t[S_\bullet D_m]\notin\cE_{Z_m}$.  Such perturbations converge to zero
	as $t\to0$, contradicting
	$\eta\in\interior(\cE_{Z_m})$.  This complete the proof.
	
	Since $\interior(\cE_{Z_m})$ is open, a sufficiently small ball around
	$\eta$ remains inside this interior.  Applying the first part at every point
	of that ball shows that the same $\mu_m$ remains uniquely maximizing.
\end{proof}
We now apply the global Baire reduction theorem  to the countable maximizable family
$\{Z_m:m\geq0\}$. Since each $Z_m$ supports only the periodic measure
$\mu_m$, Lemma~\ref{b:lem:common-interior} turns the resulting open dense
set into stable unique periodic optimization.
\begin{theorem}\label{b:thm:global-TPO}
	The set
	\begin{equation}\label{b:eq:Pfam}
		\mathcal P 
		:=\bigcup_{m=0}^{\infty}\interior(\cE_{Z_m})
	\end{equation}
	is open and dense in $\cE_{\rm orb}(X,T)$.  Every
	$\eta\in\mathcal P $ has a unique maximizing measure $\mu_m$ for a
	unique $m\geq0$.  Hence $(X,T)$ has global orbit Lipschitz TPO.  The
	maximizing periods occurring in open locking regions are unbounded.
\end{theorem}

\begin{proof}
	Openness of \eqref{b:eq:Pfam} is immediate.  By
	Theorem~\ref{b:thm:max-family} and Lemma~\ref{b:lem:EZ-closed}, the Banach space
	$\cE_{\rm orb}(X,T)$ is covered by the countable collection of closed sets
	$\cE_{Z_m}$.  The Baire category theorem implies that the union of their
	interiors is dense.  Lemma~\ref{b:lem:common-interior} gives uniqueness and
	locking.  Finally, Section~\ref{b:sec:explicit-locking} below constructs a
	non-empty locking ball corresponding to every $Z_m$.  Since
	$p_m=m+1\to\infty$, the locked maximizing periods are unbounded.
\end{proof}

\label{b:sec:explicit-locking}

Set $a_0:=0$.  For $m,n\geq0$, define
\begin{equation*}
	c_{mn}:=
	\begin{cases}
		0,&m=n,\\
		a_m+a_n,&m\neq n.
	\end{cases}
\end{equation*}
By the metric,
\begin{equation}\label{b:eq:distance-integral}
	\int_XD_m\,d\mu_n=c_{mn}.
\end{equation}
The following proposition is the form of the Theorem~\ref{thm:intro-global-locking}.
The special form of the metric and the description of the invariant
measures allow us to replace the constant $C_Q$ in the general estimate by
$1$. 
\begin{proposition}\label{b:prop:domination}
	For every $\theta\in\cE_{\rm orb}(X,T)$, every $m\geq0$, and every
	$\nu\in\M(X,T)$,
	\begin{equation*}
		|\theta_*(\nu)-\theta_*(\mu_m)|
		\leq
		\|\theta\|_{\cE_{\rm orb}}
		\int_XD_m\,d\nu.
	\end{equation*}
\end{proposition}

\begin{proof}
	First let $n\neq m$.  If $\Gamma=(\gamma_N)$ represents $\theta$, integrate
	the orbit Lipschitz inequality against the invariant product joining
	$\mu_n\times\mu_m$.  Since for any $x\in Z_{n}$, $y\in Z_{m}$ and $j\geq0$, $$
	d_X(T^jx,T^jy)=d_X(x,y)=a_n+a_m=c_{mn}.
	$$
	Therefore, 
	\begin{align*}
		\left|\frac1N\int\gamma_N\,d\mu_n
		-\frac1N\int\gamma_N\,d\mu_m\right|
		&\leq
		\frac{L_{\rm orb}(\Gamma)}N
		\sum_{j=0}^{N-1}
		\int d_X(T^jx,T^jy)\,d(\mu_n\times\mu_m)\\
		&=\frac{L_{\rm orb}(\Gamma)}N\sum_{j=0}^{N-1}\iint d_X(T^jx,T^jy)\,d\mu_{n}(x)\,d\mu_{m}(y)\\
		&=L_{\rm orb}(\Gamma)c_{mn}.
	\end{align*}
	Letting $N\to\infty$ and then taking the infimum over all representatives
	gives
	\begin{equation*}
		|\theta_*(\mu_n)-\theta_*(\mu_m)|
		\leq\|\theta\|_{\cE_{\rm orb}}c_{mn}.
	\end{equation*}
	The assertion is trivial when $n=m$. For any $\nu\in\mathcal{M}(X,T)$, by Proposition \ref{b:prop:measures}, there exists a unique sequence of $(w_n)_{n\geq 0}$ such that 
	$\nu=\sum_nw_n\mu_n$. Since $\theta_*$ is affine on $\mathcal{M}(X,T)$, we have 
	$$\theta_*(\nu)=\sum_nw_n\theta_*(\mu_n).$$
	Therefore,
	\begin{align*}
		|\theta_*(\nu)-\theta_*(\mu_{m})|=\bigg|\sum_nw_n(\theta_*(\mu_n)-\theta_*(\mu_m))\bigg|&\leq\sum_nw_n\left|\theta_*(\mu_n)-\theta_*(\mu_m)\right|\\
		&=\|\theta\|_{\cE_{\rm orb}}
		\int_XD_m\,d\nu.
	\end{align*}
\end{proof}
The following theorem is the form of the global orbit Lipschitz locking
theorem, Theorem~\ref{g:thm:global-locking}, for the present system.
Proposition~\ref{b:prop:domination} gives the domination constant $1$, so
the distance penalty produces an explicit locking ball of radius
$\varepsilon$.
\begin{theorem}\label{b:thm:locking}
	Suppose $\mu_m\in\Max(\eta)$ and let $\varepsilon>0$.  Define $\eta_{m,\varepsilon}
	:=\eta-\varepsilon[S_\bullet D_m].$
	If $\|\theta\|_{\cE_{\rm orb}}<\varepsilon$, then
	\begin{equation*}
		\Max(\eta_{m,\varepsilon}+\theta)=\{\mu_m\}.
	\end{equation*}
	In particular,
	\begin{equation*}
		B_{\cE_{\rm orb}}(\eta_{m,\varepsilon},\varepsilon)
		\subset\interior(\cE_{Z_m}).
	\end{equation*}
\end{theorem}

\begin{proof}
	For any $\nu\in\M(X,T)$, Since $\mu_m\in\Max(\eta)$  and
	Proposition~\ref{b:prop:domination} give
	\begin{align*}
		(\eta_{m,\varepsilon}+\theta)_*(\nu)
		-(\eta_{m,\varepsilon}+\theta)_*(\mu_m)
		&=(\eta_{*}(\nu)-\eta_{*}(\mu_{m}))-\varepsilon\int_{X} D_{m}\,d\nu+\theta_{*}(\nu)-\theta_{*}(\mu_{m})\\
		&\leq
		-\varepsilon\int_XD_m\,d\nu
		+\|\theta\|_{\cE_{\rm orb}}\int_XD_m\,d\nu\\
		&=
		-(\varepsilon-\|\theta\|_{\cE_{\rm orb}})
		\int_XD_m\,d\nu\\
		&\leq0.
	\end{align*}
	The last expression is strictly negative unless
	$\int D_m\,d\nu=0$.  By Proposition~\ref{b:prop:measures} and
	\eqref{b:eq:distance-integral}, the latter equality holds exactly when
	$\nu=\mu_m$.  This complete the proof.
\end{proof}

\begin{corollary}\label{b:cor:direct-density}
	The set of almost additive sequences possessing a locked unique periodic maximizing measure
	is dense in $\cE_{\rm orb}(X,T)$.
\end{corollary}

\begin{proof}
	Given $\eta\in\cE_{\rm orb}$, Theorem~\ref{b:thm:max-family} provides $m$ with
	$\mu_m\in\Max(\eta)$. Define $\eta_{m,\varepsilon}
	:=\eta-\varepsilon[S_\bullet D_m]$. Since $D_m$ is Lipschitz on $X$,
	$$
	\|\eta_{m,\varepsilon}-\eta\|_{\cE_{\rm orb}}
	\leq
	\varepsilon\bigl(\|D_m\|_\infty+\Lip(D_m)\bigr)
	\rightarrow0(\varepsilon\to0).
	$$
	Theorem~\ref{b:thm:locking} places every $\eta_{m,\varepsilon}$ in an open
	ball.  This gives a constructive proof of density, in addition to the
	Baire proof in Theorem~\ref{b:thm:global-TPO}.
\end{proof}
We end this section with two further properties of the example. The space
$\cE_{\rm orb}(X,T)$ is infinite-dimensional, and the system admits orbit
Lipschitz almost additive sequences that are not additive.

\begin{proposition}
	\label{b:prop:infinite-dimensional}
	The Banach space $\cE_{\rm orb}(X,T)$ is infinite-dimensional.
\end{proposition}

\begin{proof}
	For $m\geq1$, define
	$$
	h_m(u,t):=\mathbf 1_{Q_m}(u).
	$$
	The set $Q_m$ has distance $a_m>0$ from its complement, so $h_m$ is
	Lipschitz.  Moreover, $h_m\circ T=h_m$ and hence
	$S_nh_m=nh_m$.  Suppose a finite linear combination satisfies
	$$
	\sum_{m=1}^M c_m[S_\bullet h_m]=0
	\quad\text{in }\cE_{\rm orb}.
	$$
	Then $S_n(\sum_mc_mh_m)=n\sum_mc_mh_m$ is uniformly sublinear, which forces
	$\sum_mc_mh_m=0$.  Evaluation on $Q_m\times[0,1]$ gives $c_m=0$ for every
	$m$.  Thus the displayed additive classes are linearly independent.
\end{proof}

We finish with literal non-additive representatives.  Let $ H(u,t):=t$.

\begin{proposition}
	\label{b:prop:nonadditive}
	The sequence $\Phi^{(m)}=\{\phi_n^{(m)}\}_{n\geq1}$  defined by
	\begin{equation*}
		\phi_n^{(m)}:=-S_nD_m+(-1)^nH,\qquad n\geq1,
	\end{equation*}
	is almost additive and
	orbit Lipschitz, but it is not additive.  Its quotient class satisfies:
	\begin{enumerate}
		\item \label{b:eq:explicit-class} $[\Phi^{(m)}]=-[S_\bullet D_m]$;
		\item \label{b:eq:explicit-Max} $\Max([\Phi^{(m)}])=\{\mu_m\}$; 
		\item \label{b:eq:explicit-unit-ball} 
		$B_{\cE_{\rm orb}}([\Phi^{(m)}],1)
		\subset
		\{\eta:\Max(\eta)=\{\mu_m\}\}$.
	\end{enumerate}
\end{proposition}

\begin{proof}
	Since $D_m$ and $H$ are
	Lipschitz on $X$, each $\phi_{n}^(m)$ is orbit Lipschitz.  
	$$
	\phi_{n+s}^{(m)}-\phi_n^{(m)}-\phi_s^{(m)}\circ T^n=(-1)^{n+s}H-(-1)^nH-(-1)^sH\circ T^n,
	$$
	whose norm is bounded by $3\|H\|_\infty$, hence $\Phi^{(m)}\in\cAA_{\rm orb}(X,T)$.
	However, taking $n=s=1$ gives
	$\phi_2^{(m)}-\phi_1^{(m)}-\phi_1^{(m)}\circ T=2H+H\circ T$, which is not identically zero since $(2H+H\circ T)(u,t)=(2+\lambda)t$. Therefore $\Phi^{(m)}$ is not additive. Since $||(-1)^{n}H||_{\infty}=||H||_{\infty}\leq1$, therefore
	belongs to $\mathcal N$, proving \eqref{b:eq:explicit-class}.  Finally, applying Theorem \ref{b:thm:locking} with $\eta=0$ and $\varepsilon=1$, we have $\eta_{m,1}=-[S_\bullet D_m]=[\Phi^{(m)}]$. 
	Theorem \ref{b:thm:locking} yields $\Max([\Phi^{(m)}])=\{\mu_m\}$ and $B_{\cE_{\rm orb}}([\Phi^{(m)}],1)
	\subset
	\{\eta:\Max(\eta)=\{\mu_m\}\}$. This completes the proof.
\end{proof}

\section{Relative typical periodic optimization}
This section develops relative TPO on the fixed Lipschitz leaf
$\mathcal L_\Phi(X)$. We establish the leafwise Baire and subsystem
reduction theorems, and use periodic locking to obtain equivalent forms of
relative TPO. We also give a necessary condition, a local transfer result,
and the leafwise structural theorem. Finally, we relate relative TPO to
classical Lipschitz TPO through a Lipschitz realization condition. We first introduce two classes of Lipschitz
perturbations.
\begin{definition}
Let $(X,T)$ be a TDS, and let
$\Phi=\{\phi_n\}_{n\geq1}\subset C(X,\mathbb R)$ be an almost additive
sequence.  For each $f\in\Lip(X)$, $
\Phi^f:=\Phi+S_\bullet f
=\{\phi_n+S_nf\}_{n\geq1}.
$ For any   \(\mathcal N\subset\M(X,T)\), define
\begin{align*}
 \Lip^\Phi_{\mathcal N}(X,T)
 &:=\{f\in\Lip(X):\Max(X,T,\Phi^f)\cap\mathcal N\neq\varnothing\},\\
 \Lip^{\Phi,\subset}_{\mathcal N}(X,T)
 &:=\{f\in\Lip(X):\Max(X,T,\Phi^f)\subset\mathcal N\}.
\end{align*}
For a closed invariant set \(Z\subset X\), abbreviate
$
 \Lip^\Phi_Z=\Lip^\Phi_{\M(Z,T)},
 \Lip^{\Phi,\subset}_Z=\Lip^{\Phi,\subset}_{\M(Z,T)}.
$
\end{definition}

Thus \(f\in\Lip^\Phi_Z\) exactly when \(Z\) is maximizable for \(\Phi^f\).

\begin{remark}
The implication
\(\Lip^\Phi_Z(X,T)\neq\varnothing\Rightarrow Z\neq\varnothing\) is immediate.
In the additive cases, the converse follows
by using a constant potential.  On a fixed non-additive leaf, a constant
perturbation $f$ does not change the maximizing measures of the $\Phi$.
Consequently, non-emptiness of \(Z\) alone does not supply the converse.  
\begin{example}\label{ex 4.3}
	Let \(T\) be the identity on \([0,1]\), let
	\(\Phi=S_\bullet f\) with \(f(x)=\sqrt{x}\), thus $\Phi^g=S_\bullet(f+g)$, and let \(Z=\{0\}\).  For every \(g\in\Lip([0,1])\), one has
	\(f(x)+g(x)-(f(0)+g(0))=f(x)+g(x)-\int(f+g)\,d\delta_{0}\geq\sqrt{x}-L_{g}x>0\) for all sufficiently small \(x>0\); hence \(Z\) is not
	maximizable anywhere on the leaf \(\cL_\Phi(X)\).
\end{example}
\end{remark}
We first establish the closedness of the leafwise maximizing classes. This
property will be used in the leafwise Baire reduction.
\begin{lemma}\label{lem:closedness}
Let $(X,T)$ be a TDS, and let
$\Phi=\{\phi_n\}_{n\geq1}\subset C(X,\mathbb R)$ be an almost additive
sequence. If \(\mathcal N\subset\M(X,T)\) is closed, then
\(\Lip^\Phi_{\mathcal N}(X,T)\) is closed in \(\Lip(X)\).  In particular,
\(\Lip^\Phi_Z\) is closed for every closed invariant \(Z\).
\end{lemma}

\begin{proof}
Let \(f_j\to f\) in \(\Lip(X)\).  For every $j\geq1$, choose
$
\mu_j\in
\Max(X,T,\Phi^{f_j})\cap\mathcal N.
$
Since $\mathcal M(X,T)$ is weak$^*$ compact and $\mathcal N$ is weak$^*$ closed, it follows that
$\mu\in\mathcal N$, we may assume that
\(\mu_j\to\mu\in\mathcal N\).  For any \(\eta\in\M(X,T)\), maximality gives
$$
 \Phi_*(\mu_j)+\int f_j\,d\mu_j
 \geq \Phi_*(\eta)+\int f_j\,d\eta.
$$
Taking $j\to\infty$ and using continuity of \(\Phi_*\) together with the uniform convergence of
\(f_j\).  Then \(\mu\in\Max(X,T,\Phi^f)\).
\end{proof}
Similar, we give the definition of the leafwise maximizable family.
\begin{definition}
A collection \(\cZ\) of closed invariant subsets of \(X\) is
\emph{\(\Phi\)-leafwise maximizable} if, for every \(f\in\Lip(X)\), some
\(Z\in\cZ\) is maximizable for \(\Phi^f\).
\end{definition}
Equivalently,
\begin{equation}\label{eq:leaf-cover}
 \Lip(X)=\bigcup_{Z\in\cZ}\Lip^\Phi_Z(X,T).
\end{equation}
Corresponding to the global Baire reduction in
Theorem~\ref{g:thm:Baire-global}, we establish the following leafwise Baire
reduction on the fixed Lipschitz leaf.
\begin{theorem}\label{lem:baire}
If \(\cZ=\{Z_i\}_{i\geq1}\) is a countable \(\Phi\)-leafwise maximizable
family, then
$$
 \bigcup_{i\geq1}\interior\bigl(\Lip^\Phi_{Z_i}(X,T)\bigr)
$$
is dense in \(\Lip(X)\).
\end{theorem}
\begin{proof}
	For $i\geq1$, set
	$
	F_i:=\Lip^\Phi_{Z_i}(X,T).
	$
	Since $\mathcal Z=\{Z_i\}_{i\geq1}$ is a countable
	$\Phi$-leafwise maximizable family, \eqref{eq:leaf-cover} gives
	$
	\Lip(X)=\bigcup_{i\geq1}F_i.
	$
	By Lemma~\ref{lem:closedness}, each $F_i$ is closed in $\Lip(X)$.
	
	Let $U\subset\Lip(X)$ be an   non empty open set. Then
	$
	U=\bigcup_{i\geq1}(U\cap F_i),
	$
	where each $U\cap F_i$ is closed in $U$. Since $\Lip(X)$ is a Banach
	space, it is a Baire space. Therefore,
	there exists $i\geq1$ such that $U\cap F_i$ has non empty interior
	relative to $U$. 
	It follows that there exists a non empty open set
	$
	V\subset\interior(F_i),
	$
	and 
	$
	U\cap
	\bigcup_{i\geq1}\interior
	\bigl(\Lip^\Phi_{Z_i}(X,T)\bigr)
	\neq\varnothing.
	$
	Hence the union is dense in $\Lip(X)$.
\end{proof}
The following lemma shows that, when this set is closed and convex, the
interior of this class is equal to the interior of the stronger class.
\begin{lemma}\label{lem:common-interior}
Let $(X,T)$ be a TDS, and let
$\Phi=\{\phi_n\}_{n\geq1}\subset C(X,\mathbb R)$ be an almost additive
sequence. If \(\mathcal N\subset\M(X,T)\) is closed and convex, then
$$
 \interior\bigl(\Lip^\Phi_{\mathcal N}(X,T)\bigr)
 =
 \interior\bigl(\Lip^{\Phi,\subset}_{\mathcal N}(X,T)\bigr).
$$
\end{lemma}

\begin{proof}
It is only to prove $\interior(\Lip^\Phi_{\mathcal N}(X,T))
\subset
\interior(\Lip^{\Phi,\subset}_{\mathcal N}(X,T))$.  Suppose
\(f\in\interior(\Lip^\Phi_{\mathcal N})\) and that
\(\mu\in\Max(X,T,\Phi^f)\setminus\mathcal N\).  Strong separation Theorem \cite[Corollary~5.59]{AliprantisBorder1999} of  
\(\mu\) from the closed convex set \(\mathcal N\) gives \(\psi_0\in C(X,\mathbb{R})\) and
\(\delta>0\) such that
$
 \int\psi_0\,d\mu\geq 3\delta$, $
 \sup_{\nu\in\mathcal N}\int\psi_0\,d\nu\leq0.
$
Since $\Lip(X)$ is  dense in $C(X,\mathbb{R})$, choose
$\psi\in\Lip(X)$ such that
$
\|\psi-\psi_0\|_\infty<\delta.
$ Then
$
\int\psi\,d\mu\geq2\delta$, $
\sup_{\nu\in\mathcal N}\int\psi\,d\nu\leq\delta.
$
Consequently,
$
\int\psi\,d\mu\geq\int\psi\,d\nu+\delta
$
for every $\nu\in\mathcal N$.
For every \(t>0\) and \(\nu\in\mathcal N\),
$
 (\Phi^{f+t\psi})_*(\mu)
 \geq (\Phi^{f+t\psi})_*(\nu)+t\delta.
$
Thus \(f+t\psi\notin\Lip^\Phi_{\mathcal N}\) for every \(t>0\), contradicting
the fact that \(f\in\interior(\Lip^\Phi_{\mathcal N})\).
\end{proof}
The following theorem transfers a dense maximizing property from an
invariant subsystem to an open set of the whole Lipschitz leaf.
\begin{theorem}\label{thm:subsystem}
Let $(X,T)$ be a TDS, and let
$\Phi=\{\phi_n\}_{n\geq1}\subset C(X,\mathbb R)$ be an almost additive
sequence. Let \(Z\subset X\) be closed and invariant, and   \(U\subset\Lip(X)\) be a
non-empty open set such that \(Z\) is maximizable for \(\Phi^f\) for every
\(f\in U\).  Let \(\mathcal N\subset\M(X,T)\).  If
$
 \Lip^{\Phi|Z}_{\mathcal N\cap\M(Z,T)}(Z,T)
$
is dense in \(\Lip(Z)\), then
$
 \Lip^\Phi_{\mathcal N}(X,T)\cap U
$
is dense in \(U\).
\end{theorem}

\begin{proof}
	Let $f_0\in U$ and $\varepsilon>0$. Since $U$ is open, there exists
	$r>0$ such that
	$
	B_{\Lip}(f_0,r)\subset U. 
	$Choose
	$
	0<\eta<\min\{r,\varepsilon\}.
	$
	By the density of
	$
	\Lip^{\Phi|_Z}_{\mathcal N\cap\M(Z,T)}(Z,T|_Z)
	$
	in $\Lip(Z)$, there exists $g\in\Lip(Z)$ such that
	$\|g-f_0|_Z\|_{\Lip(Z)}<\eta$
	and
	$
	\Max\bigl(Z,T|_Z,(\Phi|_Z)^g\bigr)
	\cap\mathcal N\cap\M(Z,T)\neq\varnothing.
	$
	Choose
	\begin{equation}\label{eq:subsystem-measure}
		\mu\in
		\Max\bigl(Z,T|_Z,(\Phi|_Z)^g\bigr)
		\cap\mathcal N\cap\M(Z,T).
	\end{equation}
	Set
	$
	h_Z:=g-f_0|_Z.
	$
	By the McShane extension theorem,  
	there exists $h\in\Lip(X)$ such that
	$
	h|_Z=h_Z$
	and $\|h\|_{\Lip(X)}\leq\|h_Z\|_{\Lip(Z)}
	<\eta.$ 
	Define 	$\widetilde f:=f_0+h.$ 
	Then
	$\|\widetilde f-f_0\|_{\Lip(X)}
	=\|h\|_{\Lip(X)}<\eta<\varepsilon,$ 
	 hence $\widetilde f\in B_{\Lip}(f_0,r)\subset U.$	And
	$\widetilde f|_Z=f_0|_Z+h_Z	=g.$ 
	Then, 
	\begin{equation}\label{eq:restricted-perturbed-sequence}
		(\Phi^{\widetilde f})|_Z
		=(\Phi|_Z)^g.
	\end{equation}	
	Since $\widetilde f\in U$,  $Z$ is
	maximizable for $\Phi^{\widetilde f}$. Hence, 
	$\beta\bigl(Z,T|_Z,(\Phi^{\widetilde f})|_Z\bigr)=\beta(X,T,\Phi^{\widetilde f}).
	$ 
	By \eqref{eq:subsystem-measure} and
	\eqref{eq:restricted-perturbed-sequence}, 
	$(\Phi^{\widetilde f})_*(\mu)		=\beta\bigl(Z,T|_Z,(\Phi^{\widetilde f})|_Z\bigr)=\beta(X,T,\Phi^{\widetilde f}).$ 
	Thus
	$
	\mu\in\Max(X,T,\Phi^{\widetilde f}).
	$
	Since $\mu\in\mathcal N$, it follows that
	$
	\widetilde f\in\Lip^\Phi_{\mathcal N}(X,T)\cap U.
	$
	Hence $\Lip^\Phi_{\mathcal N}(X,T)\cap U$ is dense in $U$.
\end{proof}

Corresponding to global TPO in $\cE_{\rm orb}$, we now define periodic
optimization and relative TPO on the fixed Lipschitz leaf
$\mathcal L_\Phi(X)$.

\begin{definition}
Let $(X,T)$ be a TDS, \(f\in\Lip(X)\) has the \emph{periodic optimization property
relative to \(\Phi\)} if $$\Max(X,T,\Phi^f)=\{\mu\}$$ for some
\(\mu\in\Per(X,T)\).  It has the \emph{weak periodic optimization property}
if \(\Max(X,T,\Phi^f)\cap\Per(X,T)\neq\varnothing\).

The system \((X,T)\) has \emph{typical periodic optimization relative to
\(\Phi\)}, abbreviated  \emph{relative $\Phi$-TPO}, if an open dense subset  $P\subset\Lip(X)$ such that every $f\in\mathcal P$ has the
periodic optimization property relative to $\Phi$.
\end{definition}
To compare periodic optimization, weak periodic optimization, and locking on
the fixed leaf, we introduce the following subsets of $\Lip(X)$.
\begin{align*}
 \mathcal U_!^\Phi
 &:=\{f\in\Lip(X):\Phi^f\text{ has the periodic optimization property}\},\\
 \mathcal U_{\rm w}^\Phi
 &:=\Lip^\Phi_{\Per(X,T)}(X,T)=\{f\in\Lip(X):\Max(X,T,\Phi^f)\cap\mathcal \Per(X,T)\neq\varnothing\},\\
 \mathcal P^\Phi
 &:=\interior(\mathcal U_!^\Phi),\qquad
 \mathcal P_+^\Phi
 :=\interior(\mathcal U_{\rm w}^\Phi),\qquad
 \mathcal P_-^\Phi
 :=\bigcup_{\mu\in\Per(X,T)}
 \interior\bigl(\Lip^{\Phi,\subset}_{\{\mu\}}(X,T)\bigr).
\end{align*}
Then
\begin{equation}\label{eq:P-inclusions}
 \mathcal P_-^\Phi\subset\mathcal P^\Phi
 \subset\mathcal P_+^\Phi\subset\mathcal U_{\rm w}^\Phi.
\end{equation}

We use the following periodic locking estimate; see
\cite{BochiZhang2015}. Let $Q$ be a periodic orbit and   $\mu_Q$ be its
periodic measure. There exists $C_Q>0$ such that
\begin{equation}\label{eq:periodic-locking}
 \int g\,d\nu-\int g\,d\mu_Q
 \leq C_Q|g|_{\Lip}\int\dist(x,Q)\,d\nu(x)
\end{equation}
for all \(g\in\Lip(X)\) and \(\nu\in\M(X,T)\).  If $T$ is Lipschitz, this
estimate follows by summing over one period; see
\cite[Lemma~6.5]{HuangJenkinsonXuZhang2026}.

Corresponding to the global closure result in
Corollary \ref{g:cor:closure-equality}, periodic locking gives the following
equality of closures on the fixed Lipschitz leaf.
\begin{lemma}\label{lem:closures}
Let $(X,T)$ be a TDS, and let
$\Phi=\{\phi_n\}_{n\geq1}\subset C(X,\mathbb R)$ be an almost additive
sequence. Then
\begin{equation*}
 \cl(\mathcal P_-^\Phi)
 =\cl(\mathcal P^\Phi)
 =\cl(\mathcal P_+^\Phi)
 =\cl(\mathcal U_{\rm w}^\Phi).
\end{equation*}
\end{lemma}

\begin{proof}
By \eqref{eq:P-inclusions}, it suffices to prove
\(\mathcal U_{\rm w}^\Phi\subset\cl(\mathcal P_-^\Phi)\).  Let
\(f\in\mathcal U_{\rm w}^\Phi\), and   \(\mu_Q\) be a periodic maximizing
measure for \(\Phi^f\).  For \(\varepsilon>0\), put
$
 f_\varepsilon:=f-\varepsilon\dist(\cdot,Q).
$
Since $\mu_Q$ is the only invariant measure supported on \(Q\), $f_{\varepsilon}\in\Lip^{\Phi,\subset}_{\{\mu_{Q}\}}(X,T).$ If \(|g|_{\Lip}<\varepsilon/(2C_Q)\), then for every
\(\nu\in\M(X,T)\), using maximality of \(\mu_Q\) and
\eqref{eq:periodic-locking},
\begin{align*}
 (\Phi^{f_\varepsilon+g})_*(\nu)
 -(\Phi^{f_\varepsilon+g})_*(\mu_Q)
 &\leq -\varepsilon\int\dist(x,Q)\,d\nu
      +\int g\,d\nu-\int g\,d\mu_Q\\
 &\leq-\frac\varepsilon2\int\dist(x,Q)\,d\nu.
\end{align*}
The last inequality  is strictly negative unless \(\nu=\mu_Q\).  Hence $f_{\epsilon}+g\in\Lip^{\Phi,\subset}_{\{\mu_{Q}\}}(X,T)$, by $f_{\varepsilon}\to f$ gives the proof. 
\end{proof}
We also establish the following corollary.
\begin{corollary}\label{cor:tpo-equivalent}
The following are equivalent:
\begin{enumerate}[label=(\alph*)]
 \item \((X,T)\) has \(\Phi\)-TPO;
 \item \(\mathcal P^\Phi\) is dense in \(\Lip(X)\);
 \item \(\mathcal U_!^\Phi\) is dense in \(\Lip(X)\);
 \item \(\mathcal P_+^\Phi\) is dense in \(\Lip(X)\);
 \item \(\mathcal U_{\rm w}^\Phi\) is dense in \(\Lip(X)\);
 \item \(\mathcal P_-^\Phi\) is dense in \(\Lip(X)\).
\end{enumerate}
\end{corollary}

\begin{proof}
Use Lemma~\ref{lem:closures}, the inclusions
\eqref{eq:P-inclusions}, and the fact that \(\mathcal P^\Phi\) is the largest
open subset of \(\mathcal U_!^\Phi\).
\end{proof}
The next lemma gives a necessary condition for relative TPO. It shows that
every ergodic measure can be approximated in the weak$^*$ topology by
periodic measures.
\begin{lemma}\label{lem:periodic-density}
Suppose that TDS \((X,T)\) has \(\Phi\)-TPO, then \(\Per(X,T)\) is weak$^*$ dense in
\(\E(X,T)\).
\end{lemma}
\begin{proof}
	Let $\eta\in\E(X,T)$. By \cite{Jenkinson2006Unique}, there exists
	$h\in C(X,\mathbb{R})$ such that
	$
	\Max(X,T,S_\bullet h)=\{\eta\}.
	$
	i.e.,
	$
	\int h\,d\eta>\int h\,d\nu
	$
	for every $\nu\in\M(X,T)$ with $\nu\neq\eta$.
	
	By Remark~\ref{rem:cuneo}, there exists $f^*\in C(X,\mathbb{R})$ such that
	$
	\Phi_*(\nu)=\int f^*\,d\nu
	$
	for every $\nu\in\M(X,T)$. Hence, for every $f\in\Lip(X)$,
	$(\Phi^f)_*(\nu)
		=
		\Phi_*(\nu)+\int f\,d\nu
		=
		\int(f^*+f)\,d\nu.$	
	Since $h-f^*\in C(X,\mathbb{R})$ and $\Lip(X)$ is uniformly dense in $C(X,\mathbb{R})$, for
	each $k\geq1$, choose $g_k\in\Lip(X)$ such that
	$
	\|g_k-(h-f^*)\|_\infty<\frac1k.
	$
	Since $(X,T)$ has relative $\Phi$-TPO,
	$\mathcal U_!^\Phi$ is dense in $\Lip(X)$. Thus, for each $k\geq1$,
	we can choose
	$
	f_k\in\mathcal U_!^\Phi
	$
	such that
	$
	\|f_k-g_k\|_\infty\leq\|f_k-g_k\|_{\Lip}<\frac1k.
	$
	Therefore,
	\begin{align*}
		\|f^*+f_k-h\|_\infty
		\leq
		\|f_k-g_k\|_\infty
		+
		\|g_k-(h-f^*)\|_\infty
		<\frac2k.
	\end{align*}
	Hence
	$f^*+f_k\to h$ uniformly on $X.$ 
	Since $f_k\in\mathcal U_!^\Phi$, there exists
	$\mu_k\in\Per(X,T)$ such that
	$
	\Max(X,T,\Phi^{f_k})=\{\mu_k\}.
	$
 the measure $\mu_k$
	uniquely maximizes $f^*+f_k$. Thus, for every
	$\nu\in\M(X,T)$,
	\begin{equation*}
		(\Phi^{f_k})_*(\mu_{k})=\int(f^*+f_k)\,d\mu_k
		\geq
		\int(f^*+f_k)\,d\nu.
	\end{equation*}
	
	We now prove that
	$
	\mu_k\longrightarrow\eta
	$. Since $\M(X,T)$ is weak$^*$ compact, let
	$
	\mu_{k_j}\longrightarrow\mu
	$. Taking $j\to\infty$, we have
	$
	\int h\,d\mu\geq\int h\,d\nu
	$
	for every $\nu\in\M(X,T)$. Therefore,
	$
	\mu\in\Max(X,T,S_\bullet h).
	$
	Since $h$ has the unique maximizing measure $\eta$, we have
	$
	\mu=\eta.
	$
	Thus every weak$^*$r limit point of $\{\mu_k\}$ is equal to $\eta$.
	It follows that
	$
	\mu_k\longrightarrow\eta
	$
	in the weak$^*$ topology. Since every $\mu_k$ is periodic, hence
	$\Per(X,T)$ is weak$^*$ dense in $\E(X,T)$.
\end{proof}
\begin{corollary}\label{cor:local-transfer}
Let $(X,T)$ be a TDS, \(Z\subset X\) be closed and invariant.  Let \(U\subset\Lip(X)\) be
non-empty and open, and suppose \(Z\) is maximizable for \(\Phi^f\) for every
\(f\in U\).  If \((Z,T)\) has \((\Phi|Z)\)-TPO, then 
$
 \mathcal P^\Phi\cap U
$
is dense in \(U\).
\end{corollary}

\begin{proof}
Apply Theorem~\ref{thm:subsystem} with
\(\mathcal N=\Per(X,T)\), then
\(\mathcal N\cap\M(Z,T)=\Per(Z,T)\).  Corollary~\ref{cor:tpo-equivalent}
gives density of weak periodic optimization on \(Z\) and then shows that $\mathcal P^\Phi$ is dense in $U$.
\end{proof}
We now combine Lemma~\ref{lem:baire} with
Corollary~\ref{cor:local-transfer} to give the proof of Theorem \ref{thm:intro-leafwise-structural}. For each $i\geq1$, relative TPO on
$Z_i$ gives relative periodic optimization on the open set associated with
$Z_i$.

\begin{proof}[Proof of Theorem \ref{thm:intro-leafwise-structural}]
Set
\(U_i:=\interior(\Lip^\Phi_{Z_i}(X,T))\).  Lemma~\ref{lem:baire} says that
\(\bigcup_{i\geq0}U_i\) is dense.  For \(i\geq1\),
Corollary~\ref{cor:local-transfer} gives that
\(\mathcal P^\Phi\cap U_i\) is dense in \(U_i\).  It follows that
$
 U_0\cup\mathcal P^\Phi
$
is dense in \(\bigcup_{i\geq0}U_i\), hence dense in \(\Lip(X)\).
\end{proof}

\begin{definition}
Under the hypotheses of Theorem~\ref{thm:intro-leafwise-structural}, call \(Z_0\) the boundary
of \(\cZ\).  It is \emph{\(\Phi\)-fragile} if
$
 \interior\bigl(\Lip^\Phi_{Z_0}(X,T)\bigr)=\varnothing.
$
\end{definition}

\begin{theorem}[Countable maximizing family with relative TPO]
Suppose \(\cZ=\{Z_i:i\geq1\}\) is a countable \(\Phi\)-leafwise maximizable
family and \((Z_i,T)\) has \((\Phi|Z_i)\)-TPO for every \(i\).  Then
\((X,T)\) has \(\Phi\)-TPO.
\end{theorem}

\begin{proof}
Apply Theorem~\ref{thm:intro-leafwise-structural} with empty boundary.
\end{proof}

The relative subsystem hypotheses become classical TPO hypotheses under a
useful realization condition.

\begin{corollary}\label{cor:lip-realization}
Let $(X,T)$ be a TDS with an almost additive sequence
 $\Phi=\{\phi_n\}_{n\geq1}$. Let \(Z\subset X\) be closed and invariant.  Suppose there exists
\(f_0\in\Lip(Z)\) such that
$
 \lim_{n\to\infty}\frac1n
 \|\phi_n|_Z-S_nf_0\|_{\infty,Z}=0.
$ 
If \((Z,T)\) has classical Lipschitz TPO, then it has
\((\Phi|Z)\)-TPO.
\end{corollary}

\begin{proof}
For every \(h\in\Lip(Z)\),
$
 \Max\bigl(Z,T,(\Phi|Z)^h\bigr)
 =\Max(Z,T,f_0+h).
$
Translation \(h\mapsto f_0+h\) is a homeomorphism of \(\Lip(Z)\), Thus it preserves open
dense sets in $\Lip(Z)$. The classical TPO set therefore gives an open dense
relative TPO set. 
\end{proof}

\section{A relative TPO example}

We give a matrix example of relative TPO. In this example, the matrix
family is fixed, while the function $f$ varies in the whole space
$\Lip(\Sigma_k)$. Thus, the result is a TPO statement on one fixed
Lipschitz leaf.

	Let $k\geq2$ and let
	$
	\Sigma_k:=\{1,\ldots,k\}^{\mathbb N_0}
	$
	be the one-sided full shift with the shift map $\sigma$. Fix
	$0<\theta<1$ and use the metric
	$$
	d_\theta(x,y)
	:=
	\theta^{\min\{j\geq0:x_j\neq y_j\}}
	$$
	for $x\neq y$, and set $d_\theta(x,x)=0$.
	
	Let
	$
	C=(c_{ij})_{1\leq i,j\leq k}
	$
	be a matrix with $c_{ij}>0$. For $1\leq i\leq k$, define
	$$
	u_i:=e_i\in\mathbb R^k,
	\qquad
	v_i:=(c_{i1},\ldots,c_{ik})^{\mathsf T},
	\qquad
	A_i:=u_iv_i^{\mathsf T}.
	$$
	Then
	$
	v_i^{\mathsf T}u_j=c_{ij}>0.
	$
	For $x=(x_j)_{j\geq0}\in\Sigma_k$, set
	$$
	A^{(n)}(x):=A_{x_{n-1}}\cdots A_{x_0},
	\qquad
	\phi_n^C(x):=\log\|A^{(n)}(x)\|,
	$$
	where $\|\cdot\|$ is the Euclidean operator norm. Write
	$
	\Phi_C:=\{\phi_n^C\}_{n\geq1}.
	$
	\begin{proposition}
	$\Phi_C$ is an orbit Lipschitz almost additive sequence. Moreover,
	$(\Sigma_k,\sigma)$ has relative $\Phi_C$-TPO. Hence there exists an open
	dense set
	$$
	G\subset\Lip(\Sigma_k)
	$$
	such that, for every $f\in G$, there is a periodic orbit $Q$
	satisfying
	$$
	\M_{max}\bigl(\Sigma_k,\sigma,(\Phi_C)^f\bigr)=\{\mu_Q\}.
	$$
\end{proposition}

\begin{proof}
	Since $A_i=u_iv_i^{\mathsf T}$ and
	$v_i^{\mathsf T}u_j=c_{ij}$, we have
	$$
	A^{(n)}(x)
	=
	u_{x_{n-1}}
	\left(
	\prod_{j=1}^{n-1}c_{x_jx_{j-1}}
	\right)
	v_{x_0}^{\mathsf T}.
	$$
	Therefore,
	\begin{equation}\label{eq:leaf-matrix-formula}
		\phi_n^C(x)
		=
		\log\|u_{x_{n-1}}\|
		+\log\|v_{x_0}\|
		+\sum_{j=1}^{n-1}\log c_{x_jx_{j-1}}.
	\end{equation}
	It follows that
	\begin{align*}
		\phi_{n+m}^C(x)
		-\phi_n^C(x)
		-\phi_m^C(\sigma^nx)=\log
		\frac{c_{x_nx_{n-1}}}
		{\|v_{x_n}\|\|u_{x_{n-1}}\|}.
	\end{align*}
	The right-hand side takes only finitely many values. Hence there exists
	$C_{\Phi_C}>0$ such that
	$$
	\left|
	\phi_{n+m}^C(x)
	-\phi_n^C(x)
	-\phi_m^C(\sigma^nx)
	\right|
	\leq C_{\Phi_C}.
	$$
	Thus $\Phi_C$ is almost additive.
	
	Since the alphabet is finite, there exists $L_C>0$ such that
	$$
	|\phi_n^C(x)-\phi_n^C(y)|
	\leq
	L_C\sum_{j=0}^{n-1}
	\boldsymbol{1}_{\{x_j\neq y_j\}}.
	$$
	Moreover,
	$
	\boldsymbol{1}_{\{x_j\neq y_j\}}
	\leq d_\theta(\sigma^jx,\sigma^jy).
	$
	Consequently,
	$
	|\phi_n^C(x)-\phi_n^C(y)|
	\leq
	L_C\sum_{j=0}^{n-1}
	d_\theta(\sigma^jx,\sigma^jy).
	$
	Therefore, $\Phi_C$ is orbit Lipschitz.
	
	Define
	$$
	G_C(x):=\log c_{x_1x_0}.
	$$
	The function $G_C$ depends only on $x_0$ and $x_1$, so
	$G_C\in\Lip(\Sigma_k)$. By \eqref{eq:leaf-matrix-formula},
	$$
	\phi_n^C(x)-S_nG_C(x)
	=
	\log\|u_{x_{n-1}}\|
	+\log\|v_{x_0}\|
	-\log c_{x_nx_{n-1}}.
	$$
	The right-hand side is uniformly bounded. Hence
	$$
	\lim_{n\to\infty}
	\frac1n
	\|\phi_n^C-S_nG_C\|_\infty
	=0.
	$$
	For every $\mu\in\M(\Sigma_k,\sigma)$ and every
	$f\in\Lip(\Sigma_k)$, we therefore have
	\begin{align*}
		\bigl((\Phi_C)^f\bigr)_*(\mu)
		=
		(\Phi_C)_*(\mu)+\int f\,d\mu
		=
		\int(G_C+f)\,d\mu.
	\end{align*}
	Thus
	\begin{equation}\label{eq:matrix-max-equality}
		\M_{max}\bigl(\Sigma_k,\sigma,(\Phi_C)^f\bigr)
		=
		\M_{max}\bigl(\Sigma_k,\sigma,S_\bullet(G_C+f)\bigr).
	\end{equation}
	
	The full shift $(\Sigma_k,\sigma)$ has classical Lipschitz TPO; see
	\cite{Contreras2016,HuangJenkinsonXuZhang2026}. Let
	$\mathcal G\subset\Lip(\Sigma_k)$ be the corresponding open dense set.
	Since $G_C\in\Lip(\Sigma_k)$, the set
	$$
	G
	:=
	\{f\in\Lip(\Sigma_k):G_C+f\in\mathcal G\}
	$$
	is open and dense in $\Lip(\Sigma_k)$. For every $f\in G$,
	the right-hand side of \eqref{eq:matrix-max-equality} consists of one
	periodic measure. Therefore, $(\Sigma_k,\sigma)$ has relative
	$\Phi_C$-TPO.
\end{proof}

\begin{remark}
	The matrix $C$ is fixed in this example. The open dense set is taken in
	$\Lip(\Sigma_k)$, and each $f\in\Lip(\Sigma_k)$ gives the sequence
	$$
	(\Phi_C)^f=\Phi_C+S_\bullet f.
	$$
	Thus this is a leafwise TPO result. 	
	Also, $\Phi_C$ is not additive. Indeed, since $k\geq2$ and all entries of
	$C$ are positive,
	$
	\|v_i\|>c_{ij}
	$
	for every $i,j$. Hence the almost-additivity difference above is not zero.
	Thus the example is given by a non-additive matrix sequence.
\end{remark}

\section*{Acknowledgement}

The first author was supported by the National Natural Science Foundation of
China (No.~12271256). The second author was supported by the National Natural
Science Foundation of China (No.~12471184). The third author was supported by
the National Natural Science Foundation of China (No.~11971236) and the Qinglan
Project of Jiangsu Province of China.

\section*{Data availability}

No data was used for the research described in the article.

\section*{Conflict of interest}

The authors declare no conflict of interest.

\printauthorcontact

\end{document}